\documentclass[11pt]{article}
\usepackage{amsmath,mathrsfs}
\usepackage{amssymb}
\usepackage{amsthm}
\usepackage[all]{xy}
\usepackage{indentfirst}
\usepackage{color}
\usepackage{extarrows}
\include{extarrows}
\allowdisplaybreaks

\newtheorem{theorem}{Theorem}[section]
\newtheorem*{theorem*}{Theorem}
\newtheorem{corollary}[theorem]{Corollary}
\newtheorem{remark}[theorem]{Remark}
\newtheorem{lemma}[theorem]{Lemma}
\newtheorem{definition}[theorem]{Definition}

\newtheorem{proposition}[theorem]{Proposition}
\newtheorem{example}[theorem]{Example}
\newtheorem*{proposition*}{Proposition}

\newcommand{\R}{\mathbb{R}}

\newcommand{\be}{\begin{eqnarray*}}
\newcommand{\ee}{\end{eqnarray*}}
\newcommand{\ba}{\begin{align*}}
\newcommand{\bpm}{\begin{pmatrix}}
\newcommand{\epm}{\end{pmatrix}}

\newcommand{\bx}{\boldsymbol{x}}

\newcommand{\by}{\boldsymbol{y}}

\begin{document}
\title{Segal-Bargmann transform for generalized partial-slice monogenic functions}

\author{Zhenghua Xu$^1$ \thanks{This work was partially supported by  the Anhui Provincial Natural Science Foundation (No. 2308085MA04).}
Irene Sabadini$^2$ \thanks{This work was partially supported by PRIN 2022 {\em Real and Complex Manifolds: Geometry and Holomorphic Dynamics}.}\\
\emph{$^1$
\small School of Mathematics, Hefei University of Technology,}\\
\emph{\small  Hefei, 230601, P.R. China}\\
\emph{\small E-mail address: zhxu@hfut.edu.cn}
\\
\emph{\small $^2$ Politecnico di Milano, Dipartimento di Matematica,}\\
\emph{\small Via E. Bonardi, 9, 20133 Milano Italy} \\
\emph{\small E-mail address:   irene.sabadini@polimi.it}
}

\date{}

\maketitle

\begin{abstract}
The concept of generalized partial-slice monogenic functions has been recently introduced to include the two theories of monogenic functions and of slice monogenic functions over Clifford algebras.   The main purpose of this article is to develop the Segal-Bargmann transform and  give a Schr\"{o}dinger  representation  in the setting of generalized partial-slice monogenic functions. To this end,  the generalized partial-slice Cauchy-Kovalevskaya extension plays a crucial role.
\end{abstract}
{\bf Keywords:}\quad Functions of a hypercomplex variable;  monogenic functions; slice monogenic functions;  Segal-Bargmann transform\\
{\bf MSC (2020):}\quad  Primary: 30G35;  Secondary:  44A15, 30H20
\section{Introduction}
The Segal-Bargmann transform (or coherent state transform)  and the  Fock space (or Segal-Bargmann space) play  an essential role in various fields, among which quantum physics \cite{Hall-13}, function theory \cite{Perelomov,Zhu}, infinite-dimensional analysis \cite{Arai}. The study of the so-called Segal-Bargmann transform goes back to the works of  Bargmann and Segal in the early 1960s. The classical Segal-Bargmann transform is a unitary map from the  Schr\"{o}dinger space  of square-integrable functions to the Fock space of holomorphic functions square-integrable with respect to a Gaussian density.

Clifford analysis is an elegant generalization of holomorphic functions on the complex plane to higher dimensions and is meant as the analysis of Clifford algebra valued functions in the kernel of a generalized Cauchy-Riemann operator, called monogenic functions. Quaternionic analysis is a special case and considers the so-called regular functions from the space of quaternions to itself.

There are various generalizations of the Segal-Bargmann transform to quaternionic and Clifford analysis. The Segal-Bargmann transform has been studied in the theories of regular and monogenic functions, see \cite{Dang,Diki-Sabadini,Pena}, and also of slice regular and slice monogenic functions \cite{Cnudde-De,Diki,Diki-Ghanmi,Qian-16,Xu-Ren}. For the relation between the Segal-Bargmann transform and  the short-time Fourier transform  in Clifford Analysis, we refer the interested reader to  \cite{Bernstein-Schufmann,De-Diki}.

Very recently, the concept of generalized partial-slice monogenic functions has been developed to include the theories of of slice monogenic functions  and of monogenic functions in Clifford analysis, see \cite{Xu-Sabadini,Xu-Sabadini-3}. This paper deals with a generalization of the Segal-Bargmann transform in this framework thus providing a point of view which encompasses both the treatments. A crucial role in our approach is played by the generalized partial-slice Cauchy-Kovalevskaya extension (CK-extension) which allows to construct generalized  partial-slice monogenic functions starting from functions which are real analytic in some subset of $\mathbb R^{p+1}$ or, possibly, in the whole $\mathbb R^{p+1}$. Making use of the CK-extension we introduce and study a function which is analog to the exponential function and allows to introduce a suitable gaussian measure in our definition of the Segal-Bergmann transform. This transform has properties which generalize, by considering suitable particular cases, some results   previously demonstrated in the case of slice monogenic \cite{Qian-16} or in the case of monogenic functions \cite{Qian-17}.

The paper is organized as follows. In Section 2, we give a short introduction on generalized partial-slice monogenic functions
over Clifford algebras, in particular providing the main notions and properties  needed in the sequel. In Section 3,  we recall the generalized partial-slice Cauchy-Kovalevskaya extension previously considered in \cite{Xu-Sabadini,Xu-Sabadini-2} and we introduce and study a function which is the CK-extension of a suitable complex exponential and which plays a role in the definition of the Segal-Bargmann transform  for generalized partial-slice monogenic functions.   In Section 4, we  discuss more properties of the CK-extension, especially in the case the extended functions are entire, namely they are defined on the whole Euclidean space. In Section 5, we use the previous results to formulate the Segal-Bargmann theorem (Theorem \ref{main-Coherent-state-transforms})  in the setting of  generalized partial-slice monogenic functions, along with some  direct and interesting consequences which are compared with results in \cite{Qian-16,Qian-17}.   Section 6 is devoted to proving Theorem \ref{main-Coherent-state-transforms}. Finally, Section 7  discusses a Schr\"{o}dinger-type  representation in terms of the Segal-Bargmann transform for generalized partial-slice monogenic functions.

\section{Preliminaries}
In this section, we collect some preliminary results on Clifford algebras and monogenic functions (see \cite{Brackx,Delanghe-Sommen-Soucek, Gurlebeck}) and on  generalized partial-slice monogenic
functions  \cite{Xu-Sabadini,Xu-Sabadini-3}.
\subsection{Clifford algebras}
Given   $n\in \mathbb{N}=\{1,2,\ldots\}$,   let $\{e_1,e_2,\ldots, e_n\}$ be a standard orthonormal basis for the $n$-dimensional real Euclidean space  $\mathbb{R}^n$. A real Clifford algebra, denoted by $\mathbb{R}_{n}$, is generated  by these basis elements assuming that   $$e_i e_j + e_j e_i= -2\delta_{i,j}, \quad 1\leq i,j\leq n,$$
where $\delta_{i,j}$ is the Kronecker symbol.

As a real  vector space, the dimension of the Clifford algebra $\mathbb{R}_n$  is $2^{n}$. Every element in  $\mathbb{R}_{n}$ can be written as
 $$a=\sum_A a_Ae_A, \quad a_A\in \mathbb{R},$$
 where $e_A=e_{j_1}e_{j_2}\cdots e_{j_r}$ with $A=\{j_1,j_2, \cdots, j_k\}\subseteq\{1, 2, \cdots, n\}$ and $1\leq j_1< j_2 < \cdots < j_k \leq n,$ $1\leq k\leq n,$  $ e_\emptyset=e_0=1$.

 For $k=0,1,\ldots,n$, the real linear subspace $\mathbb{R}_n^k$ of $\mathbb{R}_n$, whose elements are called $k$-vectors, is  generated by the $\begin{pmatrix} n\\k\end{pmatrix}$ elements of the form
 $$e_A=e_{j_1}e_{j_2}\cdots e_{j_k},\quad 1\leq j_1<j_2<\cdots<j_k\leq n.$$
 In particular, we denote by $Sc(a)$ the scalar part $a_{0}=a_{\emptyset}$ of $a\in \mathbb{R}_n$.
It is readily seen that
$$\mathbb{R}_n=\mathbb{R}_n^{0}\oplus\mathbb{R}_n^{1} \oplus \cdots \oplus\mathbb{R}_n^{n}.$$

The Clifford conjugation of $a$ is  defined as
$$\overline{a} =\sum_Aa_A\overline{e_A},$$
 where $\overline{e_{j_1}\ldots e_{j_r}}=\overline{e_{j_r}}\ldots\overline{e_{j_1}},\ \overline{e_j}=-e_j,1\leq j\leq n,\ \overline{e_0}=e_0=1$.

By  the Clifford conjugation,  one may introduce an inner product and associated norm on $\mathbb{R}_n$ given by, respectively,
$$(a,b)=Sc(\overline{a}b)=Sc(b\overline{a}),$$
and
$$|a|=\sqrt{Sc(\overline{a}a)}= \sqrt{{\sum_{A}a_{A}^{2}}}.$$

An important subset of Clifford numbers $\mathbb{R}_n$ is the set of the so-called paravectors   $\mathbb{R}_n^{0} \oplus \mathbb{R}_n^{1}$. This subset will  be identified with $\mathbb{R}^{n+1}$ via the map
$$(x_0,x_1,\ldots,x_n) \longmapsto   x=x_{0}+\underline{x}=\sum_{i=0}^{n}e_ix_i.$$
Note that $\overline x=x_0-\underline{x}$. For a paravector $x\neq0$, its norm is given by $|x|=(x\overline x)^{1/2}$  and so its inverse is given by $x^{-1}=  \overline{x}|x|^{-2}.$

Denote $\mathbb{C}_n $ by  the complex Clifford algebra $\mathbb{R}_n\otimes \mathbb{C}$.  Any Clifford number  $\lambda\in \mathbb{C}_n $  may thus be written as
$$\lambda=\sum_A \lambda_Ae_A, \quad \lambda_A\in \mathbb{C},$$
or
$$\lambda= a+bi, \quad a,b\in \mathbb{R}_n.$$
In particular, we denote by  $Sc(\lambda)$ the scalar part $\lambda_{\emptyset}$ of $\lambda \in \mathbb{C}_n$ as in $\mathbb{R}_n$.

The so-called Hermitian conjugation is defined by
$$\lambda^{\dag}=\sum_A \lambda_A^{c}\overline{e_A},$$
where $\cdot^{c}$ denotes the classical complex conjugation on $\mathbb{C}$.
Explicitly,
$$(a+bi)^{\dag}=\overline{a}-\overline{b}i, \quad a,b\in \mathbb{R}_n.$$
Using the Hermitian  conjugation,  one can introduce an inner product and associated norm on $\mathbb{C}_n$ given by, respectively,
\begin{equation}\label{sharp}
(\lambda,\mu)=Sc(\lambda^{\dag}\mu)=Sc(\mu\lambda^{\dag}),
\end{equation}
and
$$|\lambda|=\sqrt{Sc(\lambda^{\dag}\lambda)}= \sqrt{{\sum_{A}|\lambda_{A}|^{2}}}.$$
Next result is well-known and its proof can be found e.g. in  \cite[Theorem 3.14]{Gurlebeck}:

  \begin{proposition}\label{norm-product}
 For $ \lambda,\mu \in \mathbb{C}_n$, one has the following:
 $$|\lambda+\mu |\leq |\lambda|+|\mu |,$$
 and
 $$|\lambda \mu |\leq 2^{\frac{n}{2}}|\lambda| |\mu |.$$
 If $ \lambda$ satisfies the relation  $\lambda\lambda^{\dag}=|\lambda|^{2}$, then it holds that
$$|\lambda \mu |=|\lambda| |\mu |. $$
  \end{proposition}

\subsection{Generalized partial-slice monogenic functions}
Throughout the paper,
given $n\in \mathbb{N}$ we split it as $n=p+q$  where
 $p\in \mathbb{N}_0=\{0\}\cup \mathbb{N}$, $q\in \mathbb{N}$, and we consider functions $f:\Omega\longrightarrow  \mathbb{C}_{p+q}$, where $\Omega\subseteq\R^{p+q+1}$ is a domain.

An element  $\bx\in\R^{p+q+1}$ will be identified with a paravector in $\mathbb{R}_{p+q}$, and we shall write it as
$$\bx=x+\underline{y} \in\R^{p+1}\oplus\R^q, \quad x=\sum_{i=0}^{p}x_i e_i,\ \underline{y}=\sum_{i=p+1}^{p+q}x_i e_i.$$

Similarly, the generalized Cauchy-Riemann operator $D_{\bx}$ is split as

\begin{equation}\label{Dxx}
D_{\bx}=D_{x} +D_{\underline{y}}, \quad D_{x} =\sum_{i=0}^{p}e_i\partial_{x_i}, \
D_{\underline{y}} =\sum_{i=p+1}^{p+q}e_i\partial_{x_i}.
\end{equation}

Denote by $\mathbb{S}$ the sphere of unit $1$-vectors in $\mathbb R^q$, whose elements $(x_{p+1},\ldots, x_{p+q})$ are identified with $\underline{y}=\sum_{i=p+1}^{p+q}x_i e_i$, i.e.
$$\mathbb{S}=\big\{\underline{y}\ :\ \underline{y}^2 =-1\big\}=\big\{\underline{y}=\sum_{i=p+1}^{p+q}x_i e_i:\sum_{i=p+1}^{p+q}x_i^{2}=1\big\}.$$
Note that, for $\underline{y}\neq0$, there exists a uniquely determined $r\in \mathbb{R}^{+}=\{x\in \mathbb{R}: x>0\}$ and $\underline{\omega}\in \mathbb{S}$, such that $\underline{y}=r\underline{\omega}$, more precisely
 $$r=|\underline{y}|, \quad \underline{\omega}=\frac{\underline{y}}{|\underline{y}|}. $$
 When $\underline{y}= 0$, set $r=0$ and $\underline{\omega}$ is not uniquely defined since $\bx=x+\underline{\omega} \cdot 0$  for every $\underline{\omega}\in \mathbb{S}$.

The upper half-space $\mathrm{H}_{\underline{\omega}}$ in $\mathbb{R}^{p+2}$ associated with $\underline{\omega}\in \mathbb{S}$ is defined by
$$\mathrm{H}_{\underline{\omega}}=\{x+r\underline{\omega}, x \in\R^{p+1}, r\geq0 \},$$
and it is clear from the previous discussion that
$$ \R^{p+q+1}=\bigcup_{\underline{\omega}\in \mathbb{S}} \mathrm{H}_{\underline{\omega}},$$
and
$$ \R^{p+1}=\bigcap_{\underline{\omega}\in \mathbb{S}} \mathrm{H}_{\underline{\omega}}.$$

In the sequel, we shall make use of the notation
$$
\Omega_{\underline{\omega}}:=\Omega\cap (\mathbb{R}^{p+1} \oplus \underline{\omega} \mathbb{R})\subseteq \mathbb{R}^{p+2},
$$
where $\Omega$ is a domain  in $\mathbb{R}^{p+q+1}$.

Now we recall the  notion of  generalized partial-slice monogenic functions following \cite{Xu-Sabadini}, but here we consider functions $\mathbb{C}_{p+q}$-valued while in \cite{Xu-Sabadini} the functions are $\mathbb{R}_{p+q}$-valued. We refer the reader to \cite{Xu-Sabadini} also for the proofs: although the context here is more general, the results below are proved with the same arguments and so they are not repeated.
\begin{definition}[Generalized partial-slice monogenic functions] \label{definition-slice-monogenic}
 Let $\Omega$ be a domain in $\mathbb{R}^{p+q+1}$. A function $f :\Omega \rightarrow \mathbb{C}_{p+q}$ is called (left)  generalized partial-slice monogenic of type $(p,q)$ if, for all $ \underline{\omega} \in \mathbb S$, its restriction $f_{\underline{\omega}}$ to $\Omega_{\underline{\omega}}\subseteq \mathbb{R}^{p+2}$  has continuous partial derivatives and  satisfies
$$D_{\underline{\omega}}f_{\underline{\omega}}(\bx):=(D_{x}+\underline{\omega}\partial_{r}) f_{\underline{\omega}}(x+r\underline{\omega})=0,$$
for all $\bx=x+r\underline{\omega} \in \Omega_{\underline{\omega}}$.
 \end{definition}
In this paper, we denote by $\mathcal {GSM}(\Omega)$ (or $\mathcal {GSM}^{L}(\Omega)$ when needed) the function class of  all  (left) generalized partial-slice monogenic functions  of type $(p,q)$  in $\Omega$.  Since in this work  we always deal with (left) generalized partial-slice monogenic functions  of type $(p,q)$, we omit to specify type $(p,q)$, as the context is clear.

Likewise, denote by $\mathcal {GSM}^{R}(\Omega)$ the set of  all right  generalized partial-slice monogenic functions of type $(p,q)$ $f:\Omega  \rightarrow \mathbb{C}_{p+q}$ which are defined by requiring that the restriction $f_{\underline{\omega}}$ to $\Omega_{\underline{\omega}}$ satisfies
$$f_{\underline{\omega}}(\bx)D_{\underline{\omega}}:=f_{\underline{\omega}} (x+r\underline{\omega})D_{x}+ \partial_{r}f_{\underline{\omega}} (x+r\underline{\omega})\underline{\omega}=0, \quad \bx=x+r\underline{\omega} \in \Omega_{\underline{\omega}},$$
for all  $ \underline{\omega} \in \mathbb S$.

\begin{remark}\label{rem32}
{\rm
 When $(p,q)=(n-1,1)$, the  notion of generalized partial-slice monogenic functions in Definition  \ref{definition-slice-monogenic} coincides with the concept of    \textit{monogenic functions}  defined in $\Omega\subseteq\mathbb{R}^{n+1}$   with values in the Clifford algebra $\mathbb{C}_{n}$, which is denoted by $\mathcal {M}(\Omega)$. For more details on the theory of monogenic functions, see e.g. \cite{Brackx,Delanghe-Sommen-Soucek,Gurlebeck}.
}
\end{remark}

\begin{remark}\label{rem32}
{\rm
When $(p,q)=(0,n)$, Definition  \ref{definition-slice-monogenic}  gives the concept  of  \textit{slice monogenic functions} defined in $\mathbb{R}^{n+1}$ and   with values in the Clifford algebra $\mathbb{C}_{n}$, which is denoted by $\mathcal{SM}(\Omega)$; see \cite{Colombo-Sabadini-Struppa-09} or \cite{Colombo-Sabadini-Struppa-11}.}
\end{remark}
Generalized partial-slice monogenic functions share some properties of slice monogenic functions, in particular, they satisfy the following version of the Splitting Lemma:
\begin{lemma} {\bf(Splitting lemma)}\label{Splitting-lemma}
Let $\Omega\subseteq \mathbb{R}^{p+q+1}$ be a domain and $f:\Omega\rightarrow \mathbb{C}_{p+q}$ be a  generalized partial-slice monogenic function. For   arbitrary, but fixed $\underline{\omega} \in\mathbb{S},$ also denoted by $I_{p+1}$, let $I_{1},I_{2},\ldots,I_{p},I_{p+2},\ldots,I_{p+q}$ be a completion to a basis of $\mathbb{R}_{p+q}$ such that $I_{r}I_{s}+I_{s}I_{r}=-2\delta_{rs}$, $r,s=1,\ldots, p+q$, $I_r\in\mathbb{R}^p$, $r=1,\ldots, p$, $I_r\in\mathbb{R}^q$, $r=p+2,\ldots,p+q$. Then there exist $2^{q-1}$ monogenic functions $F_{A}:\Omega_{\underline{\omega}} \rightarrow \mathbb{C}_{p+1}={\rm Alg}_{\mathbb{C}}\{I_1,\ldots, I_p, I_{p+1}=\underline{\omega}\}$  such that
$$ f_{\underline{\omega}}(\bx_p+r\underline{\omega})=\sum_{A=\{i_{1},\ldots, i_{s}\} \subseteq \{p+2,\ldots,p+q\}}F_{A}(\bx_p+r\underline{\omega})I_{A}, \quad \bx_p+r\underline{\omega}\in \Omega_{\underline{\omega}},$$
where $I_{A}=I_{i_{1}}\cdots I_{i_{s}},  A=\{i_{1},\ldots, i_{s}\} \subseteq \{p+2,\ldots,p+q\}$ with $i_{1}<\cdots<i_{s}$, and $I_{\emptyset}=1$ when $A=\emptyset$.
\end{lemma}

In the sequel, to prove some results we need  more terminology and the Identity Principle which are stated below.
\begin{definition} \label{slice-domain}
 Let $\Omega$ be a domain in $\mathbb{R}^{p+q+1}$.

1.   $\Omega$ is called  slice domain if $\Omega\cap\mathbb R^{p+1}\neq\emptyset$  and $\Omega_{\underline{\omega}}$ is a domain in $\mathbb{R}^{p+2}$ for every  $\underline{\omega}\in \mathbb{S}$.

2.   $\Omega$   is called  partially  symmetric with respect to $\mathbb R^{p+1}$ (p-symmetric for short) if, for   $x\in\R^{p+1}, r \in \mathbb R^{+},$ and $ \underline{\omega}  \in \mathbb S$,
$$\bx=x+r\underline{\omega} \in \Omega\Longrightarrow [\bx]:=x+r \mathbb S=\{x+r \underline{\omega}, \ \  \underline{\omega}\in \mathbb S\} \subseteq \Omega. $$
 \end{definition}
Note that the condition of that $\Omega$ is p-symmetric is an $SO(q)$-invariant domain with respect to $\mathbb R^{p+1}$.

Denote by  $B(\by,\rho)=\{\bx \in \mathbb{R}^{p+q+1}: |\by-\bx|<\rho \}$  the ball centered in $\by\in \mathbb{R}^{p+q+1}$ with radius $\rho>0$.
Obviously, $B(\by,\rho)$ are p-symmetric slice domains for all $\by\in \mathbb{R}^{p+1}$.

We recall   an identity theorem for generalized partial-slice monogenic functions over slice domains. Denote by $\mathcal{Z}_{f}(\Omega)$  the zero set of the function $f:\Omega\subseteq\mathbb{R}^{n+1}\rightarrow \mathbb{C}_{n}$.

\begin{theorem}  {\bf(Identity theorem)}\label{Identity-theorem}
Let $\Omega\subseteq \mathbb{R}^{p+q+1}$ be a  slice domain and $f,g:\Omega\rightarrow \mathbb{C}_{p+q}$ be   generalized partial-slice monogenic functions.
If there is an imaginary $\underline{\omega}  \in \mathbb S $ such that $f=g$ on a   $(p+1)$-dimensional smooth manifold in $\Omega_{\underline{\omega}}$, then $f\equiv g$ in  $\Omega$.
\end{theorem}

\section{CK-extension}

In this section, we review the generalized partial-slice Cauchy-Kovalevskaya extension already given  in \cite{Xu-Sabadini,Xu-Sabadini-2} and we discuss  more properties and examples, which shall be used in the formulation of the Segal-Bargmann transform for generalized partial-slice
monogenic functions.

\begin{definition}[CK-extension]\label{Slice-Cauchy-Kovalevska-extension}
Let   $\Omega_{0}$ be a domain in  $\mathbb{R}^{p+1}$ and let $\Omega_{0}^{\ast}$ be the domain  defined by
$$\Omega_{0}^{\ast}=\{ x+\underline{y}:x\in \Omega_{0}, \underline{y}\in \mathbb{R}^{q}\}.$$
Given a  real analytic function $f_{0}: \Omega_{0}  \to \mathbb{C}_{p+q}$, if there is a generalized partial-slice monogenic function $f^{\ast}$ in $\Omega \subseteq \Omega_{0}^{\ast}$ with $\Omega_{0} \subset \Omega$ and $f^{\ast}(x)=f_0(x)$,
the function $f^{\ast}$  is called the generalized partial-slice Cauchy-Kovalevskaya extension (CK-extension, for short) of the real analytic function $f_{0}$.
\end{definition}

\begin{theorem}(\cite[Theorem 4.2]{Xu-Sabadini-2}) \label{slice-Cauchy-Kovalevsky-extension}
Let   $\Omega_{0}$ be a domain in  $\mathbb{R}^{p+1}$ and  $f_{0}: \Omega_{0}  \to \mathbb{C}_{p+q}$ be a real analytic function.  Then the function given by
\begin{equation}\label{CKCC}
CK[f_{0}](\bx)= \exp (\underline{y} D_{x} ) f_{0}(x)
=\sum_{k=0}^{+\infty} \frac{1}{k!}( \underline{y} D_{x})^{k}f_{0}(x), \quad \bx=x+ \underline{y},
\end{equation}
is the CK-extension in a  p-symmetric slice domain $\Omega \subseteq \Omega_{0}^{\ast}$ with $\Omega_{0} \subset \Omega$.
\end{theorem}
We omit the proof of this result since it repeats verbatim the proof of  \cite[Theorem 4.2]{Xu-Sabadini-2}, the only difference being the values of the function under consideration.

\begin{remark}\label{Cauchy-Kovalevska-extension-special-cases}
When $(p,q)=(0,n), (n-1,1)$, the  CK-extension given by Definition \ref{Slice-Cauchy-Kovalevska-extension} reduces to the standard slice  monogenic extension (denoted by $M_{s}$), and  to  a  modified version of the  Cauchy-Kovalevskaya extension for monogenic functions, respectively.  See  \cite[Definition 14.5]{Brackx} for the classical Cauchy-Kovalevskaya extension for monogenic functions and \cite{Delanghe-Sommen-Soucek} for the case of monogenic functions with values in the complexified Clifford algebra. The case  $(p,q)=(0,1)$, when the functions have values in the real Clifford algebra $\mathbb R_1$ namely $\mathbb C$, reduces to the standard  holomorphic extension, or analytic  continuation, denoted by $\mathcal{C}$. We shall use
$\mathcal{C}$ also to denote the analytic continuation for functions with values in $\mathbb C_1$, namely the algebra of bicomplex numbers.
\end{remark}

\begin{remark}\label{Cauchy-Kovalevska-extension-odd-even}
Note that   $CK[f_{0}]$ has the following  useful decomposition
$$ CK[f_{0}](\bx)=\sum_{k=0}^{+\infty} \frac{r^{2k}}{(2k)!}  (-\Delta_{x})^{k} f_{0}(x)
+\underline{\omega} D_{x} \sum_{k=0}^{+\infty} \frac{r^{2k+1}}{(2k+1)!}(-\Delta_{x}) ^{k}f_{0}(x),
$$
where $\bx=x+\underline{y}$ with $\underline{y}=r\underline{\omega}, r=|\underline{y}|$. See also Remark 4.4 in \cite{Xu-Sabadini-2}.
\end{remark}

\begin{example}\label{example-polynomial}
Let $k\in \mathbb{N}_0$ and $p_k(x, \xi )=(\langle x, \xi \rangle)^{k}$, $x,\xi\in\mathbb{R}^{p+1}$.
Define
$$p_k(\bx, \xi):=CK  [p_k(\cdot, \xi )](\bx).$$
Then $$p_k(\bx, \xi)=(\langle x, \xi \rangle+\underline{y} \xi)^{k},  \quad \bx=x+\underline{y}.$$
\end{example}
\begin{proof}
Note that
$$ D_{x}(\langle x, \xi \rangle)^{k}=k(\langle x, \xi \rangle)^{k-1}\xi,$$
then we have $$( \underline{y} D_{x})^{2}(\langle x, \xi \rangle)^{k}=k \underline{y} D_{x}(\langle x, \xi \rangle)^{k-1} \underline{y}\xi=k(k-1)(\langle x, \xi \rangle)^{k-2} (\underline{y}\xi)^{2}.$$
 Assuming that, for $l< k$, the formula
$$( \underline{y} D_{x})^{l}(\langle x, \xi \rangle)^{k}=k (k-1)\ldots(k-l+1)(\langle x, \xi \rangle)^{k-l} (\underline{y}\xi)^l$$
is valid, it is immediate to show that it holds also for $l+1$. When $l+1=k$ the right hand side equals $k!$ and so $( \underline{y} D_{x})^{\ell}(\langle x, \xi \rangle)^{k}=0$ for all $\ell> k$.
Hence we obtain
\begin{eqnarray*}
CK[p_k(\cdot, \xi )](\bx)
 &=&\sum_{l=0}^{+\infty} \frac{1}{l!}( \underline{y} D_{x})^{l}(\langle x, \xi \rangle)^{k}
 \\
&=&\sum_{l=0}^{k} \frac{k(k-1)\cdots(k-l+1)}{l!} (\langle x, \xi \rangle)^{k-l} (\underline{y}\xi)^{l}
\\
&=&(\langle x, \xi \rangle+\underline{y} \xi)^{k},
\end{eqnarray*}
as asserted.
\end{proof}

\begin{remark}
Note that  real-analytic functions $f$ on $\mathbb{R}^{n}$, without further assumptions, may not  extend to a monogenic function  on $\mathbb{R}^{n+1}$ (compare with \cite[Corollary 14.6 and Theorem 14.8]{Brackx}). This fact is well-known in complex analysis and in fact, if we consider $(p,q)=(0,1)$ and the real analytic function  $f_{0}(x)=\frac{1}{1+x^{2}}$  on $\mathbb{R}$, which has the holomorphic  extension $f(z)=\frac{1}{1+z^{2}}$ on the  neighbourhood  $\{z=x+yi\in \mathbb{C}: x\in \mathbb{R}, -1<y<1 \}\subset \mathbb{C}$ of $\mathbb{R}$, it is immediate that it cannot be extended  analytically to the whole complex space  $\mathbb{C}$.
\end{remark}

Let $\mathrm{k}=(k_0,k_1,\ldots,k_{p})\in\mathbb{N}_{0}^{p+1}, k=|\mathrm{k}| = k_0+k_1+\ldots+k_p$, and  $f:\Omega_{0}\subseteq \mathbb{R}^{p+1}\rightarrow \mathbb{C}_{p+q}$ be a function with continuous $k$-times partial derivatives,  we denote by $\partial_{\mathrm{k}}$ the   partial derivative given by
$$ \partial_{\mathrm{k}}f(x)=\frac{\partial^{k} }{\partial_{x_{0}}^{k_0}\partial_{x_{1}}^{k_1}\cdots\partial_{x_{p}}^{k_p} }f(x).
$$

Below we denote by $C^{\infty} (\Omega_{0}, \mathbb{C}_{p+q})$ the right $\mathbb{C}_{p+q}$-module of $C^{\infty}$ functions defined on $\Omega_{0}$.
\begin{proposition}\label{Ck-entire-condition}
Let   $\Omega_{0}$ be a domain in  $\mathbb{R}^{p+1}$ and  $f_{0}\in C^{\infty} (\Omega_{0}, \mathbb{C}_{p+q})$. If $ \partial_{\mathrm{k}}f_{0}(x) $ is uniformly bounded on $\Omega_{0}$ for all  $\mathrm{k}\in\mathbb{N}_{0}^{p+1}$ such that $|\mathrm k|$ is large enough, then
$CK[f_{0}]$ is well-defined on  $\Omega_{0}^{\ast}$.
\end{proposition}

\begin{proof}
By assumption, there exist a constant $C>0$ and $N\in \mathbb{N}$ such that
$$ |\partial_{\mathrm{k}}f_{0}(x) |\leq C, \quad x\in \Omega_{0}, |\mathrm{k}|\geq N.$$
As a consequence, both the series
$$\sum_{l=N}^{+\infty} \frac{r^{2l}}{(2l)!}  (-\Delta_{x})^{l} f_{0}(x), \ \ \ \  \sum_{l=N}^{+\infty} \frac{r^{2l+1}}{(2l+1)!}(-\Delta_{x})^{l}  D_{x}f_{0}(x)$$  are  absolutely convergent on  $\Omega_{0}$ for all $r\geq0$.
Hence, by Remark \ref{Cauchy-Kovalevska-extension-odd-even}, the series \eqref{CKCC} defining $CK[f_{0}]$ is convergent on the whole $\Omega_{0}^{\ast}$. The proof is complete.
\end{proof}

Let us denote by  $\mathcal{S}(\mathbb{R}^{p+1})$   the space of real-valued  Schwartz functions on $\mathbb{R}^{p+1}$.  We have the following consequence of the previous result:
\begin{proposition}\label{Ck-entire-condition-Schwarz}
If $f_{0}\in \mathcal{S}(\mathbb{R}^{p+1})\otimes \mathbb{C}_{p+q}$, then  $CK[f_{0}]$  is the unique extension of $f_{0}$ to  the whole space  $\mathbb{R}^{p+q+1}$ preserving generalized partial-slice monogenicity.
\end{proposition}
\begin{proof}
It is immediate that the function class $\mathcal{S}(\mathbb{R}^{p+1})\otimes \mathbb{C}_{p+q}$ satisfies the condition in Proposition \ref{Ck-entire-condition}.  Hence, by Theorem    \ref{Identity-theorem} we get the assertion.
\end{proof}

\begin{proposition}\label{example-e}
For each $\xi\in \mathbb{R}^{p+1}$, define the left generalized partial-slice   monogenic functions in $\bx\in \mathbb{R}^{p+q+1}$ as
$$e(\bx, \xi):=CK  [e^{i\langle \cdot , \xi \rangle}](\bx).$$
Then $e(\bx, \xi)$ has the following  expression
\begin{equation}\label{Representation-e}
e(\bx, \xi)=\big(\cosh(|\underline{y}| |\xi|)+ i \frac{\underline{y}}{|\underline{y}|}\frac{\xi}{|\xi|} \sinh(|\underline{y}| |\xi|)\big)e^{i\langle x, \xi \rangle},  \quad \bx=x+\underline{y}.
\end{equation}
\end{proposition}
\begin{proof}
First, note that  $e(\cdot, \cdot)$ is well defined for $ \mathbb{R}^{p+q+1}\times \mathbb{R}^{p+1}$ from Proposition \ref{Ck-entire-condition}.

Let $\bx=x+\underline{y}$ with $\underline{y}=r\underline{\omega}, r=|\underline{y}|$.
From the basic formulas
 $$\Delta_{x} e^{i\langle x, \xi \rangle}=-|\xi|^{2}e^{i\langle x, \xi \rangle},\ \ \  D_{x}e^{i\langle x, \xi \rangle}=i\xi e^{i\langle x, \xi \rangle},$$
 we have by Remark \ref{Cauchy-Kovalevska-extension-odd-even}
\begin{eqnarray*}
  CK  [e^{i\langle x, \xi \rangle}](\bx)
 &=&\sum_{k=0}^{+\infty} \frac{r^{2k}}{(2k)!}  (-\Delta_{x})^{k} e^{i\langle x, \xi \rangle}
+\underline{\omega} D_{x} \sum_{k=0}^{+\infty} \frac{r^{2k+1}}{(2k+1)!}(-\Delta_{x}) ^{k} e^{i\langle x, \xi \rangle}
 \\
&=&\sum_{k=0}^{+\infty} \frac{(r|\xi|)^{2k}}{(2k)!}   e^{i\langle x, \xi \rangle}
+\underline{\omega} D_{x} \sum_{k=0}^{+\infty} \frac{r^{2k+1}|\xi|^{2k}}{(2k+1)!}e^{i\langle x, \xi \rangle}
\\
&=&\sum_{k=0}^{+\infty} \frac{(r|\xi|)^{2k}}{(2k)!}   e^{i\langle x, \xi \rangle}
+ i \underline{\omega}\xi \sum_{k=0}^{+\infty} \frac{r^{2k+1}}{(2k+1)!}( |\xi|) ^{2k}e^{i\langle x, \xi \rangle}
\\
&=&\Big(\sum_{k=0}^{+\infty} \frac{(r|\xi|)^{2k}}{(2k)!}
+ i \underline{\omega}\frac{\xi}{|\xi|} \sum_{k=0}^{+\infty} \frac{( r|\xi|)^{2k+1}}{(2k+1)!} \Big)e^{i\langle x, \xi \rangle}
\\
&=&\big(\cosh(r|\xi|)+ i \underline{\omega}\frac{\xi}{|\xi|} \sinh(r|\xi|) \big)e^{i\langle x, \xi \rangle},
\end{eqnarray*}
which completes the proof.
\end{proof}

\begin{proposition}\label{proposition-e}
The function $e(\bx, \xi)$ satisfies the basic identities:
\begin{equation}\label{proposition-e-1}
e(\bx, \xi)=e(x, \xi) e(\underline{y}, \xi)= e(\underline{y}, \xi)e(x, \xi), \end{equation}
\begin{equation}\label{proposition-e-2}
e(\bx, \xi) ^{\dag}=e(-x, \xi) e(\underline{y}, \xi)=e(x, -\xi) e(\underline{y}, \xi),\end{equation}
\begin{equation}\label{proposition-e-3}
e(\bx, \xi) ^{\dag}e(\bx, \xi)=e( \underline{y}, \xi) ^{\dag}e( \underline{y}, \xi) =e(2 \underline{y}, \xi)=  e( \underline{y},2 \xi),\end{equation}
and
\begin{equation}\label{proposition-e-4}
(e(\bx, \xi) )^{k}= e(k\bx, \xi)  =  e(\bx,k\xi), \quad k \in \mathbb{N}.\end{equation}
\end{proposition}

\begin{proof}
From the definition in formula \eqref{Representation-e}, we immediately have that
$$e(x,\xi)= e^{i\langle x,\xi\rangle},$$
$$e(\underline{y},\xi)= \cosh(|\underline{y}| |\xi|)+ i \frac{\underline{y}}{|\underline{y}|}\frac{\xi}{|\xi|} \sinh(|\underline{y}| |\xi|)\big),
$$
and so we obtain (\ref{proposition-e-1}).  Then (\ref{proposition-e-1}) implies the validity of (\ref{proposition-e-2}), (\ref{proposition-e-3}), (\ref{proposition-e-4}) if we can prove the following two formulas
\begin{equation}\label{proposition-e-proof-1}
e( \underline{y}, \xi) ^{\dag}=e( \underline{y}, \xi),\end{equation}
and
\begin{equation}\label{proposition-e-proof-2}
e( \underline{y}, \xi)^{2}=e(2 \underline{y}, \xi)=  e( \underline{y},2 \xi).\end{equation}
In fact, observing that
$$(i  \underline{y} \xi)^{\dag}=-i\overline{\xi}\overline{ \underline{y}} = i\overline{\xi}  \underline{y} = i  \underline{y} \xi,  $$
we can get (\ref{proposition-e-proof-1})
and
$$(i  \underline{y} \xi)^{\dag}i  \underline{y} \xi = (i  \underline{y} \xi)^{2}=(  \underline{y} \xi)^{\dag}  \underline{y} \xi=\overline{ \xi} \overline{\underline{y} } \ \underline{y} \xi = |\underline{y}|^{2} |\xi|^{2}, $$
which implies (\ref{proposition-e-proof-2}).
The proof is complete.
\end{proof}

\section{CK-extension to $\mathbb{R}^{p+q+1}$}
\label{Sect4}

In this section, we continue  discussing  the   CK-extension in the case the functions can be extended to the whole $\mathbb{R}^{p+q+1}$, and so they extend as entire functions. In particular, we shall show that all real analytic functions in $\mathbb{R}^{p+1}$ with  Taylor
expansion  convergent as  a multiple power series in $\mathbb{R}^{p+1} $  possess  CK-extension to whole space $\mathbb{R}^{p+q+1}$; see Theorem \ref{slice-Cauchy-Kovalevsky-extension-entire} below. This result is crucial to  study the Segal-Bargmann transform for generalized partial-slice monogenic functions.
 To this end, we first recall Fueter polynomials  and  Taylor series for generalized partial-slice monogenic functions, see \cite{Xu-Sabadini}.

Consider a multi-index $\mathrm{k}=(k_0,k_1,\ldots,k_{p})\in \mathbb{N}_{0}^{p+1}$, set  $k=|\mathrm{k}|=k_0+k_1+\cdots+k_{p}$, and $\mathrm{k}!=k_0 !k_1 !\cdots k_{p}!$. Let  $\bx'=(x_0,x_1,\ldots,x_p,r) \in \mathbb{R}^{p+2}$, and set $\bx'_{\diamond}=(x_0,x_1,\ldots,x_p,-r)$. Let $\underline{\eta} \in \mathbb{S}$ be arbitrary but fixed.  The so-called Fueter variables are defined as
   \begin{equation}\label{Fueter-variables}
z_{\ell}=  x_{\ell}+r \underline{\eta} e_\ell, \ \ell=0,1,\ldots,p.
\end{equation}
An immediate calculation shows that for any $\underline{\eta}\in\mathbb S$ we have
$$(\sum_{i=0}^{p}e_i\partial_{x_i}+\underline{\eta}\partial_{r}) z_{\ell}= e_\ell+\underline{\eta}^2e_\ell=0,  \ \ell=0,1,\ldots,p,
$$
and so the variable $z_\ell$ is  generalized partial-slice monogenic.
\begin{definition}\label{definition-Fueter}
For  $\mathrm{k}=(k_0,k_1,\ldots,k_{p})\in \mathbb{N}_{0}^{p+1}$, let   $(j_1,j_2,\ldots, j_k) \in \{0,1,\ldots,p\}^{k}$   be an alignment with the number of $0$ in the alignment is $k_0$, the number of $1$ is $k_1$ and the number of $p$ is $k_{p}$, where $k=|\mathrm{k}|$.
For $\underline{\eta} \in \mathbb{S}$, define
$$\mathcal{P}_{ \underline{\eta},\mathrm{k}}  ( \bx')=\frac{1}{k!}\sum_{ \sigma} z_{\sigma(j_1)} z_{\sigma(j_2)}\cdots z_{\sigma(j_k)},$$
where the sum is computed over the $\dfrac{k!}{\mathrm{k}!}$ different permutations $\sigma$ of $(j_1,j_2,\ldots, j_k)$. When $\mathrm{k}=(0,\ldots,0)=0$, we set $\mathcal{P}_{ \underline{\eta},0}(\bx')=1$.

Similarly, we can define $\mathcal{P}^{R}_{ \underline{\eta},\mathrm{k}}  ( \bx')$ when $z_{\ell}$ are replaced by $z_{\ell}^{R}$.
\end{definition}
Note that for $p=0$, $\mathrm{k} =k\in \mathbb{N}_0$,
$\mathcal{P}_{ \underline{\eta},\mathrm{k}}  ( \bx')=\mathcal{P}_{ \underline{\eta},k}  (x_0,r)=\frac{1}{k!} (x_{0}+r \underline{\eta})^{k}$.

Definition \ref{definition-Fueter} allows to establish Taylor series for generalized partial-slice monogenic function, see \cite{Xu-Sabadini} for the proof.
\begin{theorem}[Taylor series]\label{Taylor-theorem}
Let $f: B=B(0,\rho)\rightarrow \mathbb{R}_{p+q}$ be a generalized partial-slice monogenic function. For  any $\bx\in B$, we have
$$f(\bx)=  \sum_{k=0}^{+\infty} \sum_{|\mathrm{k}|=k} \mathcal{P}_{\mathrm{k}}(\bx)  \partial_{\mathrm{k}} f (0),  $$
where the series converges uniformly on compact subsets of $B$.
\end{theorem}
We also recall the next result from \cite{Xu-Sabadini}:
\begin{proposition} \label{Ck-P}
Let $\mathrm{k}=(k_0,k_1,\ldots,k_{p})\in \mathbb{N}_{0}^{p+1}$ and $x^{\mathrm k}:=x_0^{k_0}\ldots x_p^{k_p}$, then
  $CK[x^{\mathrm{k}}](\bx)= \mathrm{k}!\mathcal{P}_{ \mathrm{k}}(\bx).$     In particular, $z_{\ell}=  CK[x_{\ell}](x+\underline{\eta} r), \ \ell=0,1,\ldots,p.$
\end{proposition}
The following two results can be easily obtained and will be useful in the sequel.
 \begin{proposition}\label{estimate-Fueter}
 It holds that
$$|\mathcal{P}_{ \underline{\eta},\mathrm{k}}  ( \bx')|\leq \frac{1}{k!}\sum_{ \sigma} |z_{\sigma(j_1)}|| z_{\sigma(j_2)}| \cdots |z_{\sigma(j_k)} | \leq   \dfrac{  1}{\mathrm{k}!} |z_0|^{k_0}  \cdots |z_p|^{k_p}.$$
  \end{proposition}
\begin{proof}
Note that $ z_{\ell}\overline{z_{\ell}}=  x_{\ell}^{2}+r^{2}=|z_{\ell}|^{2}$. Hence,  the  estimate for the Fueter polynomials follows by Proposition  \ref{norm-product}.
\end{proof}
 \begin{proposition}\label{Futer-valued}
Let $f_{0}: \mathbb{R}^{p+1} \to\mathbb{R}$ be a  polynomial. Then we have $CK[f_{0}](\bx)=F_1(\bx')+\underline{\omega} F_2(\bx') \in \mathcal{GSM}^{L}(\mathbb{R}^{p+q+1})$ with $F_1(\bx')\in \mathbb{R}, F_2(\bx')\in \mathrm{span}_{\mathbb{R}}\{1,e_{1},\ldots,e_p \}$.
\end{proposition}
\begin{proof}
Remark \ref{Cauchy-Kovalevska-extension-odd-even} and the fact that $(-\Delta_x)^k$ is  a real operator immediately give  the result.
\end{proof}
Proposition \ref{Futer-valued} shows  that  when dealing with functions whose restriction to $\mathbb R^{p+1}$ are real valued, we have that $F_1 +\underline{\omega} F_2 \in \mathbb{R}_{p+q}^{0}\oplus\mathbb{R}_{p+q}^{1} \oplus \mathbb{R}_{p+q}^{2}$, hence we have that
\[
\begin{split}(F_1 +\underline{\omega} F_2) \overline{(F_1 +\underline{\omega} F_2)}&=
(F_1 +\underline{\omega} F_2) {(F_1 -\overline{F_2}\underline{\omega} )}\\
&=
(F_1 +\underline{\omega} F_2) {(F_1 -\underline{\omega}{F_2} )}\\
&=
 F_1^{2}- \underline{\omega} F_2\underline{\omega} F_2\\
 &=F_1^{2}+|F_2|^{2}=|F_1 +\underline{\omega} F_2|^{2},
 \end{split}
 \]
which gives,  by Proposition  \ref{Ck-P},
\begin{equation}\label{Fueter-norm}
\mathcal{P}_{ \mathrm{k}}  ( \bx ) \overline{\mathcal{P}_{ \mathrm{k}}  ( \bx )} =|\mathcal{P}_{ \mathrm{k}}  ( \bx )|^{2}.
\end{equation}
This fact shall be useful in  proving the next result, namely Theorem \ref{slice-Cauchy-Kovalevsky-extension}, whose proof also makes use of a version of Weierstrass theorem for generalized partial-slice monogenic functions as follows.

\begin{lemma}\label{Weierstrass}
  Let $\{f_{m}\}_{m\in \mathbb{N}}$ be a sequence of  generalized partial-slice monogenic functions in a p-symmetric slice
domain  $\Omega$  that converges to a function $f$ uniformly on compact sets in   $\Omega$. Then $f\in \mathcal{GSM}(\Omega)$.
\end{lemma}
\begin{proof}
Recall that any function $f:\ \Omega\subseteq \mathbb{R}^{p+q+1}\to\mathbb{C}_{p+q}$ can be written as
$$f(\bx)=\sum_{A=\{i_{1},\ldots, i_{s}\} \subseteq \{p+2,\ldots,p+q\}}F_{A}(\bx)I_{A}, \quad \bx\in \Omega.$$
It is now enough to prove that $F_A$ are monogenic on $\Omega_{\underline{\omega}}$.
By the splitting lemma (Lemma \ref{Splitting-lemma}), there exist   monogenic functions $F_{m,A}:\Omega_{\underline{\omega}} \rightarrow \mathbb{C}_{p+1}={\rm Alg}_{\mathbb{C}}\{I_1,\ldots, I_p, $ $I_{p+1}=\underline{\omega}\}$  such that
$$ f_m (x+r\underline{\omega})=\sum_{A=\{i_{1},\ldots, i_{s}\} \subseteq \{p+2,\ldots,p+q\}}F_{m,A}(\bx_p+r\underline{\omega})I_{A}, \quad \bx_p+r\underline{\omega}\in \Omega_{\underline{\omega}}.$$
Clearly, $F_{m,A} \rightarrow F_{A}$ uniformly on compact sets
in  $\Omega_{\underline{\omega}}$, so that $F_{A}$ must be monogenic  in  $\Omega_{\underline{\omega}}$, as desired.
\end{proof}
Let $\rho>0$ be given, denote by $C(\rho)$ the open $(p+1)$-dimensional cube centred
at the origin with edge $2\rho$, i.e.
$$C(\rho)=\{ \bx_p: |x_\ell|<\rho,  \ell=0,1,\ldots,p\}.  $$
Let us consider the p-symmetric slice domain
 $$D(\rho)=\{\bx_p+\underline{\bx}_q:  x_\ell^{2}+|\underline{\bx}_q|^{2}<\rho^{2}, \ell=0,1,\ldots,p\}.$$
Observe that $D(\rho)\cap \mathbb{R}^{p+1}=C(\rho)\subset D(\rho)$.

Now let us show the rationality of the CK-extension  to $D(\rho)$ of a real analytic function defined in  $C(\rho)$, which is a variant of \cite[Theorem 14.2]{Brackx}.
\begin{theorem}\label{slice-Cauchy-Kovalevsky-extension}
Let  $f_{0}: C(\rho)  \to \mathbb{C}_{p+q}$ be a real analytic function such that its Taylor
series about $0$ converges as a multiple power series in $C(\rho)$.  Then the function given by
$$CK[f_{0}](\bx)= \exp (\underline{y} D_{x} ) f_{0}(x)
=\sum_{k=0}^{+\infty} \frac{1}{k!}( \underline{y} D_{x})^{k}f_{0}(x), \quad \bx=x+ \underline{y},$$
is well-defined and gives the CK-extension of $f_0$ in $D(\rho)$.
\end{theorem}
\begin{proof} By assumption, the function
$$f_{0}(x)=\sum_{k=0}^{+\infty} \frac{1}{k!}  \sum_{i_1,\ldots, i_k=0}^p  x_{i_1} \ldots  x_{i_k}  \frac{\partial^k}{\partial_{x_{i_1}}\ldots \partial_{x_{i_k}}}f_0(0)
=\sum_{k=0}^{+\infty} \sum_{|\mathrm{k}|=k} \frac{x^{\mathrm{k}}}{\mathrm{k}!}
  \partial_{\mathrm{k}} f_{0}(0)$$
 converges as a multiple power series in $C(\rho)$. This means that for every $0<\varrho<\rho$
$$\sum_{k=M}^{N}  \sum_{|\mathrm{k}|=k}  \frac{ \varrho^{k}}{\mathrm{k}!}
  |\partial_{\mathrm{k}} f_{0}(0)| \rightarrow 0,\ {\mbox {as}} \ \inf \{N,M\} \rightarrow\infty.$$
Denote
  $$S_{N}(\bx)=\sum_{k=0}^{N} \sum_{|\mathrm{k}|=k} \mathcal{P}_{\mathrm{k}}(\bx)
  \partial_{\mathrm{k}} f_{0}(0).$$
Hence, by (\ref{Fueter-norm}), Proposition \ref{norm-product} and  Proposition \ref{estimate-Fueter},
  \begin{eqnarray*} \label{ReprPk}
 |S_{N}(\bx)-S_{M}(\bx)|
 & \leq& \sum_{k=M}^{N} \sum_{|\mathrm{k}|=k} |\mathcal{P}_{\mathrm{k}}(\bx) \partial_{\mathrm{k}} f_{0}(0)|
 \\
&=&\sum_{k=M}^{N} \sum_{|\mathrm{k}|=k} |\mathcal{P}_{\mathrm{k}}(\bx)|  |\partial_{\mathrm{k}} f_{0}(0)|
 \\
& \leq &\sum_{k=M}^{N} \sum_{|\mathrm{k}|=k} \dfrac{  1}{\mathrm{k}!} |z_0|^{k_0}  \cdots |z_p|^{k_p} |\partial_{\mathrm{k}} f_{0}(0)|.
\end{eqnarray*}
Now for $\bx\in \overline{D(\varrho)}$, we have
$$|S_{N}(\bx)-S_{M}(\bx)|\leq\sum_{k=M}^{N} \sum_{|\mathrm{k}|=k} \dfrac{  \varrho ^{k}}{\mathrm{k}!}   |\partial_{\mathrm{k}} f_{0}(0)|
\rightarrow 0,\ {\mbox {as}} \ \inf \{N,M\} \rightarrow\infty.$$
 Hence, by Lemma \ref{Weierstrass}, there  exists a function $f^{\ast}\in \mathcal{GSM}^{L}(D(\rho))$ such that
$$f^{\ast}(\bx)=\lim_{N\rightarrow\infty} S_{N}(\bx)=\sum_{k=0}^{\infty } \sum_{|\mathrm{k}|=k} \mathcal{P}_{\mathrm{k}}(\bx)
  \partial_{\mathrm{k}} f_{0}(0),$$
as desired.
  \end{proof}

Set $B_{p+1}(\by,\rho)=B(\by,\rho)\cap \mathbb{R}^{p+1}$. Note that we have the inclusion relations
$$C(\frac{\rho}{\sqrt{p+1}})\subseteq B_{p+1}(0,\rho), \ B(0,\rho)\subseteq D(\rho).$$
Hence, Theorem \ref{slice-Cauchy-Kovalevsky-extension} gives that

  \begin{corollary}\label{slice-Cauchy-Kovalevsky-extension-ball}
Let  $f_{0}:B_{p+1}(0,\rho)  \to \mathbb{C}_{p+q}$ be a real analytic function such that its Taylor
series about $0$ converges as a multiple power series in $B_{p+1}(0,\rho)$.  Then the function
$CK[f_{0}]$ is well-defined and gives the CK-extension of $f_0$ in $B(0,\frac{\rho}{\sqrt{p+1}})$.
\end{corollary}

Letting  $\rho\rightarrow +\infty$ in  Corollary \ref{slice-Cauchy-Kovalevsky-extension-ball}, we obtain by Theorem \ref{Identity-theorem}  the following result.
\begin{corollary} \label{slice-Cauchy-Kovalevsky-extension-entire}
Let  $f_{0}:  \mathbb{R}^{p+1}  \to \mathbb{C}_{p+q}$ be a real analytic function such that its Taylor
series about $0$ converges as a multiple power series in $\mathbb{R}^{p+1} $.  Then
$CK[f_{0}]$ is the unique CK-extension of $f_0$ in $\mathbb{R}^{p+q+1} $.
\end{corollary}

From  Theorem   \ref{Taylor-theorem} and  Corollary \ref{slice-Cauchy-Kovalevsky-extension-entire}, we get directly that
 \begin{corollary}\label{slice-Cauchy-Kovalevsky-extension-entire-Identity}
 Let $f$  be a generalized partial-slice monogenic  function on $\mathbb{R}^{p+q+1}$ with  $f(0,\underline{y})=f_{0}(\underline{y})$. Then
 $f(x,\underline{y})=CK[f_{0}(x)](x,\underline{y}).$
  \end{corollary}

\section{Segal-Bargmann transforms}
In this section, we study the Segal-Bargmann transform in the context of generalized partial-slice monogenic functions. To this end, we first recall briefly the classical Segal-Bargmann transform for holomorphic functions.

\subsection{The holomorphic case}
Recall  the heat equation in $\mathbb{R}^{d}$
$$ \frac{1}{2} \Delta_{x} [f](x,t)=\partial_{t}[f](x,t), \quad (x,t)\in \mathbb{R}^{d}\times \mathbb{R}^{+},$$
subject to the initial condition $$\lim_{t\searrow0}f(x,0)=f_{0}(x),$$
where $f_{0}: \mathbb{R}^{d}\rightarrow \mathbb{R}$ is a square integrable function   with respect to Euclidean metric, i.e.
$f_{0}\in L^{2}(\mathbb{R}^{d}, dx)$, and  the limit is in the norm topology of $L^{2}(\mathbb{R}^{d}, dx)$.

It is a well-known that the solution to the heat equation is given by
$$f(x,t)= \exp(\frac{t}{2}\Delta_{x})[f_{0}]:=\int_{\mathbb{R}^{d}} \rho_{t}(x-\xi)f_{0}(\xi)d\xi,$$
where $\rho_{t}(x)=\rho(x,t)=\frac{1}{(2\pi t)^{d/2}}e^{-\frac{|x|^{2}}{2t}}$ is the heat kernel in $\mathbb{R}^{d}$.

Note that the function $\rho_{t}$ admits analytic continuation to $\mathbb{C}^{d}$ given by
$$(\rho_{t})_{\mathbb{C}}(z)=\rho(z,t)=\frac{1}{(2\pi t)^{d/2}}e^{-\frac{z^{2}}{2t}},$$
where $z=(z_1,\ldots, z_d)\in \mathbb{C}^{d}, z^{2}=\sum_{i=1}^{d}z_{i}^{2}$.
\begin{definition}
Define the map $U: L^{2}(\mathbb{R}^{d}, dx)\rightarrow \mathcal{H}(\mathbb{C}^{d})$ as
\begin{equation}\label{Segal-Bargmann-transform-def-holo}
 U[f](z) =  \frac{1}{ (2\pi)^{d/2} }\int_{\mathbb{R}^{d} }e^{-\frac{(z- x)^{2}}{2}}f(x)dx,
\end{equation}
where $\mathcal{H}(\mathbb{C}^{d})$ denotes the space of entire holomorphic functions on $\mathbb{C}^{d}$.
The map $U$ is called to be  the Segal-Bargmann transform (or coherent state transform) (cf. \cite{Hall}).
\end{definition}

 Denote by $\mathcal{A}(\mathbb{R}^{d})$  the space of  real analytic complex-valued functions on $\mathbb{R}^{d}$ with unique analytic
continuation to $\mathcal{H}(\mathbb{C}^{d})$,  and by $\mathcal{C}$   the analytic continuation from   $\mathbb{R}^{d}$ to $\mathbb{C}^{d}$.
In fact, the analytic continuation $\mathcal{C}$  on $f\in \mathcal{A}(\mathbb{R}^{d})$ can be written in the form
$$\mathcal{C}[f](z)=\exp (i \sum_{j=1}^{d}y_{j} \partial _{x_{j}})f(x), \quad z=x+yi\in \mathbb{C}^{d}.  $$

The celebrated Segal-Bargmann theorem reads as follows; see \cite[Theorem 6.4]{Hall}\label{Hall}.
\begin{theorem}\label{Segal-Bargmann-theorem}
The transform $U$ in the diagram
$$\xymatrix{
L^{2}(\mathbb{R}^{d}, dx) \ar[dr]_{U} \ar[r]^{\exp\frac{\Delta_{x}}{2}}
                & \widetilde{\mathcal{A}}(\mathbb{R}^{d}) \ar[d]^{\mathcal{C}}  \\
                & \mathcal{H}L^{2}(\mathbb{C}^{d},\frac{1}{ \pi^{d/2}}e^{-y^{2}}dxdy)          }$$
is a unitary isomorphism, where   $\widetilde{\mathcal{A}}(\mathbb{R}^{d}) \subset\mathcal{A}(\mathbb{R}^{d})$  denotes the image of $L^{2}(\mathbb{R}^{d}, dx)$ by the operator $\exp\frac{\Delta_{x}}{2}$ and
 $$\mathcal{H}L^{2}(\mathbb{C}^{d},\frac{1}{ \pi^{n/2} }e^{-y^{2}}dxdy)=\mathcal{H}(\mathbb{C}^{d})\cap L^{2}(\mathbb{C}^{d}, \frac{1}{ \pi^{n/2}}e^{-y^{2}}dxdy).$$
 \end{theorem}
The result in Theorem  \ref{Hall} is a slight modification of results of Segal \cite{Segal} and Bargmann \cite{Bargmann}.
For more details, we refer the reader to \cite{Hall}.

\subsection{The generalized partial-slice monogenic case}

 To obtain  an analogous   of Theorem \ref{Segal-Bargmann-theorem} in the setting of  generalized partial-slice monogenic functions, we  replace the  analytic continuation $\mathcal C$ by the extension which is the suitable one in this context, namely the CK-extension and we consider the diagram:
$$\xymatrix{
L^{2}(\mathbb{R}^{p+1}, dx)\otimes \mathbb{C}_{p+q} \ar[dr]_{?} \ar[r]^{\exp\frac{\Delta_{x}}{2}}
                &  \widetilde{\mathcal{A}}(\mathbb{R}^{p+1})\otimes \mathbb{C}_{p+q} \ar[d]^{CK}   \\
                & \mathcal{GSM}(\mathbb{R}^{p+q+1}).           }$$
where   $\widetilde{\mathcal{A}}(\mathbb{R}^{p+1})$  denotes the image of $L^{2}(\mathbb{R}^{p+1}, dx)$ by the operator $\exp\frac{\Delta_{x}}{2}$ as in Theorem \ref{Segal-Bargmann-theorem}.

To this end this, we first need  to show the first issue
\begin{equation}\label{inclusion}
\widetilde{\mathcal{A}}(\mathbb{R}^{p+1})\otimes \mathbb{C}_{p+q} \subset \mathcal{AGSM}(\mathbb{R}^{p+1}),
\end{equation}
where  $\mathcal{AGSM}(\mathbb{R}^{p+1})$  denotes  the set of all $\mathbb{C}_{p+q}$-valued  real analytic  functions defined in $\mathbb{R}^{p+1}$ with unique generalized partial-slice monogenic  extensions to $\mathbb{R}^{p+q+1}$.

By Theorem \ref{Segal-Bargmann-theorem}, it follows that $\widetilde{\mathcal{A}}(\mathbb{R}^{p+1}) \subset\mathcal{A}(\mathbb{R}^{p+1})$, and so all functions in $\widetilde{\mathcal{A}}(\mathbb{R}^{p+1})$
are   real analytic  and such that their Taylor series about $0$ converge as a multiple power series in $\mathbb{R}^{p+1}$. Hence, Corollary \ref{slice-Cauchy-Kovalevsky-extension-entire} yields the inclusion in (\ref{inclusion}).

Next, one has to give a suitable definition of transform, generalizing the classic Segal-Bargmann transform in (\ref{Segal-Bargmann-transform-def-holo}).

Observe that
$$\Delta_{x} e^{i\langle x, \xi \rangle}=-|\xi|^{2}e^{i\langle x, \xi \rangle},$$
$$\exp(\frac{\Delta_{x}}{2}) e^{i\langle x, \xi \rangle}=  e^{-\frac{|\xi|^{2}}{2}} e^{i\langle x, \xi \rangle},$$
and
$$e^{-\frac{|x|^{2}}{2}}=  \frac{1}{ (2\pi)^{(p+1)/2} }\int_{\mathbb{R}^{p+1} } e^{-\frac{|\xi|^{2}}{2}} e^{i\langle x, \xi \rangle} d\xi.$$
Hence,
$$CK[\rho_{1}](\bx)=\exp(\underline{y}D_{x}) \rho_{1}(x)=\frac{1}{ (2\pi)^{p+1 } }\int_{\mathbb{R}^{p+1} } e(\bx, \xi)e^{-\frac{|\xi|^{2}}{2}}d\xi. $$

 Inspired by the formula above, we give the following definition in the setting of generalized partial-slice
monogenic functions.

\begin{definition} \label{definition-Segal-Bargmann-transform}
 For $f\in L^2(\mathbb{R}^{p+1}, dx)\otimes \mathbb{C}_{p+q}$, we define  the Segal-Bargmann  transform (or coherent state transform ) $U_{p,q}$ of $f$ as
 $$U_{p,q} [f] (\bx)=\frac{1}{ (2\pi)^{p+1} }\int_{\mathbb{R}^{p+1} } \Big( \int_{\mathbb{R}^{p+1} }e(\bx, \xi)  e^{-\frac{|\xi|^{2}}{2}} e^{-i\langle \zeta, \xi \rangle}d\xi \Big)  f(\zeta) d\zeta.$$
 \end{definition}
\begin{lemma} \label{partial-slice-monogenic}
If  $f\in L^2(\mathbb{R}^{p+1}, dx)\otimes \mathbb{C}_{p+q}$,  then $ U_{p,q} [f]$ is a generalized partial-slice monogenic function on $\mathbb{R}^{p+q+1}$.
\end{lemma}
 \begin{proof}
First of all, we note that the gaussian factor in the integrand ensures the integrability of all of its derivatives. Then, for every $\underline{\omega}\in\mathbb S$, we get
$$D_{\underline{\omega}}U_{p,q} [f] (\bx)=\frac{1}{ (2\pi)^{p+1} }\int_{\mathbb{R}^{p+1} } \Big( \int_{\mathbb{R}^{p+1} }(D_{\underline{\omega}}e(\bx, \xi))  e^{-\frac{|\xi|^{2}}{2}} e^{-i\langle \zeta, \xi \rangle}d\xi \Big)  f(\zeta) d\zeta=0,$$
i.e., $ U_{p,q} [f]$ is generalized partial-slice monogenic  on $\mathbb{R}^{p+q+1}$.
 \end{proof}

 \begin{lemma}\label{remark-Segal-Bargmann-transform}
 For $f\in \mathcal{S}(\mathbb{R}^{p+1})\otimes \mathbb{C}_{p+q}$,    $ U_{p,q} [f]$  can be rewritten as
 $$U_{p,q} [f]  =CK \circ \exp(\frac{\Delta_{x}}{2}) [f],$$
and
\begin{eqnarray}\label{FCK}
 U_{p,q}[f] (\bx)
=\frac{1}{ (2\pi)^{(p+1)/2} }\int_{\mathbb{R}^{p+1} }  e^{-\frac{|\xi|^{2}}{2}} e(\bx, \xi) \widehat{f}(\xi)d\xi,
\end{eqnarray}
where $\bx=x+\underline{y}$ and $\widehat{f}$ denotes the standard Fourier transform of $f$ in $\mathbb{R}^{p+1}$.
 \end{lemma}
\begin{proof}
For $f\in \mathcal{S}(\mathbb{R}^{p+1})\otimes \mathbb{C}_{p+q}$, it holds that for $\bx=x+\underline{y}\in \mathbb{R}^{p+q+1}$
\begin{eqnarray*}
 U_{p,q} [f] (\bx)
&=&\frac{1}{ (2\pi)^{p+1} }\int_{\mathbb{R}^{p+1} } \Big( \int_{\mathbb{R}^{p+1} }e(\bx, \xi)  e^{-\frac{|\xi|^{2}}{2}} e^{-i\langle \zeta, \xi \rangle}f(\zeta)d\xi \Big)  d\zeta
\\
&=&\frac{1}{ (2\pi)^{p+1} }\int_{\mathbb{R}^{p+1} }e(\bx, \xi) e^{-\frac{|\xi|^{2}}{2}}   d\xi
\int_{\mathbb{R}^{p+1} } e^{-i\langle \xi,\zeta \rangle} f(\zeta) d\zeta
\\
 &=&\frac{1}{ (2\pi)^{(p+1)/2} }\int_{\mathbb{R}^{p+1} }e(\bx, \xi) e^{-\frac{|\xi|^{2}}{2}}   \widehat{f}(\xi)d\xi
\\
 &=&\frac{1}{ (2\pi)^{(p+1)/2} }\int_{\mathbb{R}^{p+1} } CK  [e^{i\langle x, \xi \rangle}e^{-\frac{|\xi|^{2}}{2}}](\bx) \widehat{f}(\xi)d\xi
\\
 &=&CK    \Big[ \frac{1}{ (2\pi)^{(p+1)/2} }\int_{\mathbb{R}^{p+1} }e^{i\langle x, \xi \rangle}e^{-\frac{|\xi|^{2}}{2}} \widehat{f}(\xi)d\xi  \Big](\bx)
 \\
 &=& CK  \circ \exp\frac{\Delta_{x}}{2} \Big[ \frac{1}{ (2\pi)^{(p+1)/2} }\int_{\mathbb{R}^{p+1} }e^{i\langle x, \xi \rangle} \widehat{f}(\xi)d\xi  \Big](\bx)
 \\
 &=&CK  \circ \exp\frac{\Delta_{x}}{2} [ f](\bx),
 \end{eqnarray*}
as asserted.
\end{proof}

Note that \eqref{FCK} can be further rewritten as:
\begin{eqnarray*}
 U_{p,q}[f] (\bx)
&=&\frac{1}{ (2\pi)^{(p+1)/2} }\int_{\mathbb{R}^{p+1} }  e^{-\frac{|\xi|^{2}}{2}} e(\bx, \xi) \widehat{f}(\xi)d\xi
\\
&\xlongequal{\text{(\ref{Representation-e})}}&\frac{1}{ (2\pi)^{(p+1)/2} }\int_{\mathbb{R}^{p+1} }\cosh(|\underline{y}| |\xi|)   e^{-\frac{|\xi|^{2}}{2}} e^{i\langle x, \xi \rangle} \widehat{f} (\xi)d\xi
 \\
 && +i \frac{\underline{y}}{|\underline{y}|} \frac{1}{ (2\pi)^{(p+1)/2} }\int_{\mathbb{R}^{p+1} } \frac{\xi}{|\xi|} \sinh(|\underline{y}| |\xi|)  e^{-\frac{|\xi|^{2}}{2}} e^{i\langle x, \xi \rangle}  \widehat{f}(\xi)d\xi,
\end{eqnarray*}
where $\bx=x+\underline{y}$.
 \begin{proposition}
Let $f\in \mathcal{S}(\mathbb{R}^{p+1})\otimes \mathbb{C}_{p+q}$  be such that
\begin{equation}\label{condition-S}
e(\underline{y}, \cdot) \widehat{f}(\cdot) \in \mathcal{S}(\mathbb{R}^{p+1})\otimes \mathbb{C}_{p+q}, \quad \text{for\ all\ }\underline{y}\in \mathbb{R}^{q}.\end{equation}
Then
$$F(\bx)=\frac{1}{ (2\pi)^{(p+1)/2} }\int_{\mathbb{R}^{p+1} }  e(\bx, \xi)  \widehat{f}(\xi)d\xi
=\frac{1}{ (2\pi)^{(p+1)/2} }\int_{\mathbb{R}^{p+1} }  e^{i\langle x, \xi \rangle} e(\underline{y}, \xi)  \widehat{f}(\xi)d\xi$$
 defines a generalized partial-slice monogenic  function  on $\mathbb{R}^{p+q+1}$ satisfying $F(x,\underline{0})=f(x)$.
  \end{proposition}
  \begin{proof}
 For $f\in \mathcal{S}(\mathbb{R}^{p+1})\otimes \mathbb{C}_{p+q}$, we have by definition
$$F(x,\underline{0})
=\frac{1}{ (2\pi)^{(p+1)/2} }\int_{\mathbb{R}^{p+1} }  e^{i\langle x, \xi \rangle}    \widehat{f}(\xi)d\xi=f(x).$$
Note that, for each $\xi\in \mathbb{R}^{p+1}$, $e( \cdot, \xi)$ is left generalized partial-slice   monogenic function  on $ \mathbb{R}^{p+q+1}$ by Proposition \ref{example-e}.  The hypothesis  (\ref{condition-S}) ensures that  $F(\bx)$ is well-defined on   $\mathbb{R}^{p+q+1}$ and we  can differentiate under the integral  in  $F$ and thus we can conclude that $F$  is generalized partial-slice monogenic.
\end{proof}

Now we can state the  main result in this section.

 \begin{theorem}\label{main-Coherent-state-transforms}
 For $f\in L^2(\mathbb{R}^{p+1}, dx)\otimes \mathbb{C}_{p+q}$, we have
$$U_{p,q} [f] (\bx) =CK \circ \exp(\frac{\Delta_{x}}{2}) [f](\bx). $$
Then,  in the commutative diagram
$$\xymatrix{
L^{2}(\mathbb{R}^{p+1}, dx)\otimes \mathbb{C}_{p+q} \ar[dr]_{U_{p,q}} \ar[r]^{\exp\frac{\Delta_{x}}{2}}
                & \widetilde{\mathcal{A}}(\mathbb{R}^{p+1})\otimes \mathbb{C}_{p+q} \ar[d]^{CK}  \\
                & \mathcal{GSM} L^{2}(\mathbb{R}^{p+q+1}, d\mu),           }$$
the map $U_{p,q}:L^{2}(\mathbb{R}^{p+1}, dx)\otimes \mathbb{C}_{p+q}\rightarrow \mathcal{GSM} L^{2}(\mathbb{R}^{p+q+1}, d\mu)$  is a unitary isomorphism, where $\widetilde{\mathcal{A}}(\mathbb{R}^{p+1})$  denotes the image of $L^{2}(\mathbb{R}^{p+1}, dx)$ by the operator $\exp\frac{\Delta_{x}}{2}$, and the set
 $$\mathcal{GSM} L^{2}(\mathbb{R}^{p+q+1}, d\mu) = \mathcal{GSM}(\mathbb{R}^{p+q+1}) \cap (L^{2}(\mathbb{R}^{p+q+1}, d\mu)\otimes \mathbb{C}_{p+q})$$ is the  Hilbert space of generalized partial-slice monogenic   functions on $\mathbb{R}^{p+q+1}$ which are $L^{2}(\mathbb{R}^{p+q+1}, d\mu)\otimes \mathbb{C}_{p+q}$ with respect to
the measure
 $$d\mu(\bx) = \frac{2}{\sqrt{\pi}}\frac{1}{ |\mathbb{S}|}  \frac{e^{-|\underline{y}|^{2}}}{|\underline{y}|^{q-1}} dxd\underline{y}, \quad \bx=x+\underline{y},$$
with $|\mathbb{S}|=2{\pi^{\frac q2}}/{\Gamma\left(\frac q2\right)}$.
 \end{theorem}

The proof of Theorem \ref{main-Coherent-state-transforms} is postponed to the next section. Here we discuss some direct consequences of this result.

 \begin{corollary}
The subspace $\mathcal{GSM} L^{2}(\mathbb{R}^{p+q+1}, d\mu)$ of generalized partial-slice monogenic    functions, which are in $L^{2}(\mathbb{R}^{p+q+1}, d\mu)\otimes \mathbb{C}_{p+q}$,  is a closed subspace of $L^{2}(\mathbb{R}^{p+q+1}, $ $d\mu)\otimes \mathbb{C}_{p+q}$.
  \end{corollary}

A  characterization of the range of the heat operator is obtained in terms of generalized partial-slice monogenic
functions is obtained below. It is analogous to that one for monogenic functions \cite[Corollary 3.3]{Qian-17}.
   \begin{corollary}
A real analytic function $F$ on  $\mathbb{R}^{p+1}$ is of the form
$$F=\exp(\frac{\Delta_{x}}{2})[f], \quad  f\in L^{2}(\mathbb{R}^{p+1}, dx)\otimes \mathbb{C}_{p+q},$$
if and only if  its CK-extension to $\mathbb{R}^{p+q+1}$ exists and is $d\mu$-square-integrable, i.e.
$$CK[f]\in \mathcal{GSM} L^{2}(\mathbb{R}^{p+q+1}, d\mu).$$
\end{corollary}

The following two corollaries are the counterparts of analogous results obtained in \cite{Qian-16, Qian-17}.
In  the special case $(p,q)=(0,n)$, we obtain a counterpart of \cite[Theorem 4.3]{Qian-16}.
 \begin{corollary}\label{main-Coherent-state-transforms-0n}
 For $f\in L^2(\mathbb{R}, dx)\otimes \mathbb{C}_{n}$, define
  $$U_{s} [f] (\bx)=\frac{1}{ 2\pi  }\int_{\mathbb{R}  } \Big( \int_{\mathbb{R} }e(\bx, \xi)  e^{-\frac{|\xi|^{2}}{2}} e^{-i\langle \zeta, \xi \rangle}d\xi \Big)  f(\zeta) d\zeta,$$
  where $\bx=x+\underline{y}\in \mathbb{R} \oplus \mathbb{R}^{n}.$

Then,  in the commutative diagram
$$\xymatrix{
L^{2}(\mathbb{R}, dx)\otimes \mathbb{C}_{n} \ar[dr]_{U_s} \ar[r]^{\exp\frac{\Delta_{x}}{2}}
                & \widetilde{\mathcal{A}}(\mathbb{R})\otimes \mathbb{C}_{n} \ar[d]^{M_s}  \\
                & \mathcal{SM} L^{2}(\mathbb{R}^{n+1}, d\mu),           }$$
the map $U_{s}:L^{2}(\mathbb{R}, dx)\otimes \mathbb{C}_{n}\rightarrow \mathcal{SM} L^{2}(\mathbb{R}^{n+1}, d\mu)$  is a unitary isomorphism, where $$\mathcal{SM} L^{2}(\mathbb{R}^{n+1}, d\mu) = \mathcal{SM}(\mathbb{R}^{n+1}) \cap (L^{2}(\mathbb{R}^{n+1}, d\mu)\otimes \mathbb{C}_{n})$$ is the Hilbert space of slice monogenic   functions on $\mathbb{R}^{n+1}$ which are $L^{2}(\mathbb{R}^{n+1}, d\mu)\otimes \mathbb{C}_{n}$ with respect to
the measure
 $$d\mu(\bx) = \frac{2}{\sqrt{\pi}}\frac{1}{ |\mathbb{S}|}  \frac{e^{-|\underline{y}|^{2}}}{|\underline{y}|^{n-1}} dxd\underline{y}, \quad \bx=x+\underline{y},$$
 with $|\mathbb{S}|=2 {\pi^{\frac n2}}/{\Gamma\left(\frac n2\right)}$.
 \end{corollary}

In the special case $(p,q)=(n-1,1)$, we obtain a result equivalent to \cite[Theorem 3.1]{Qian-17}.
 \begin{corollary}\label{main-Coherent-state-transforms-0n}
 For $f\in L^2(\mathbb{R}^{n}, dx)\otimes \mathbb{C}_{n}$, define
  $$V [f] (\bx)=\frac{1}{ (2\pi)^{n}  }\int_{\mathbb{R}^{n}  } \Big( \int_{\mathbb{R}^{n} }e(\bx, \xi)  e^{-\frac{|\xi|^{2}}{2}} e^{-i\langle \zeta, \xi \rangle}d\xi \Big)  f(\zeta) d\zeta,$$
  where $\bx=x+\underline{y}\in \mathbb{R}^{n} \oplus \mathbb{R}.$

Then,  in the commutative diagram
$$\xymatrix{
L^{2}(\mathbb{R}^{n}, dx)\otimes \mathbb{C}_{n} \ar[dr]_{V} \ar[r]^{\exp\frac{\Delta_{x}}{2}}
                & \widetilde{\mathcal{A}}(\mathbb{R}^{n})\otimes \mathbb{C}_{n} \ar[d]^{CK}  \\
                & \mathcal{M} L^{2}(\mathbb{R}^{n+1}, d\mu),           }$$
the map $V:L^{2}(\mathbb{R}^{n}, dx)\otimes \mathbb{C}_{n}\rightarrow \mathcal{M} L^{2}(\mathbb{R}^{n+1}, d\mu)$  is a unitary isomorphism, where $$\mathcal{M} L^{2}(\mathbb{R}^{n+1}, d\mu) = \mathcal{M}(\mathbb{R}^{n+1}) \cap (L^{2}(\mathbb{R}^{n+1}, d\mu)\otimes \mathbb{C}_{n})$$ is the Hilbert space of  monogenic   functions on $\mathbb{R}^{n+1}$ which are $L^{2}(\mathbb{R}^{n+1}, d\mu)\otimes \mathbb{C}_{n}$ with respect to
the measure
 $$d\mu(\bx) = \frac{1}{\sqrt{\pi}}  e^{-|\underline{y}|^{2}} dxd\underline{y}, \quad \bx=x+\underline{y}.$$ \end{corollary}

 \section{Proof of Theorem \ref{main-Coherent-state-transforms}}
Following \cite{Qian-17}, to  prove our main Theorem \ref{main-Coherent-state-transforms} we need some lemmas.

Before to state and prove the next result, we recall the well-known Plancherel theorem.
 \begin{theorem}[Plancherel theorem]\label{Plancherel-theorem}
For $f,h\in \mathcal{S}(\mathbb{R}^{p+1})\otimes \mathbb{C}_{p+q}$, we have
$$\langle \widehat{f}, \widehat{h} \rangle_{L^{2}(\mathbb{R}^{p+1}, d\xi)\otimes \mathbb{C}_{p+q}}
  = \langle  f, h \rangle_{L^{2}(\mathbb{R}^{p+1}, dx)\otimes \mathbb{C}_{p+q}}.$$
\end{theorem}
\begin{lemma}\label{isometry-Schwarz}
For $f,h\in \mathcal{S}(\mathbb{R}^{p+1})\otimes \mathbb{C}_{p+q}$, we have the equality
$$\langle U_{p,q} [f], U_{p,q} [h]\rangle_{L^{2}(\mathbb{R}^{p+q+1}, d\mu)\otimes \mathbb{C}_{p+q}}=\langle  f, h \rangle_{L^{2}(\mathbb{R}^{p+1}, dx)\otimes \mathbb{C}_{p+q}}. $$
\end{lemma}

 \begin{proof}
 First of all, we deduce that by Proposition \ref{proposition-e}
\begin{equation}\label{proposition-e-proof}
e(\bx, \xi) ^{\dag}e(\bx,  \zeta)= e(x,  \zeta-\xi) e(\underline{y}, \xi)e(\underline{y},\zeta).
\end{equation}
 Note  that below we shall use the inner product in \eqref{sharp}. For $f,h\in \mathcal{S}(\mathbb{R}^{p+1})\otimes \mathbb{C}_{p+q}$ we have that
\begin{eqnarray} \label{proposition-integral-0-proof}
&& \int_{\mathbb{R}^{p+q+1} } \frac{e^{-|\underline{y}|^{2}}}{|\underline{y}|^{q-1}} \Big(   \widehat{f}(\xi),
    i \frac{\underline{y}}{|\underline{y}|}\frac{\xi}{|\xi|} \sinh(2|\underline{y}| |\xi|)    \widehat{h}(\xi) \Big) e^{- |\xi|^{2} } d\xi   d\underline{y} \notag
    \\
    &=& \int_{\mathbb{R}^{p+1} }  \Big(   \widehat{f}(\xi),
    i  \Big(\int_{\mathbb{R}^{q} } \frac{e^{-|\underline{y}|^{2}}}{|\underline{y}|^{q-1}}   \frac{\sinh(2|\underline{y}| |\xi|)}{|\underline{y}||\xi|}    \underline{y}   d\underline{y}\Big)  \xi  \widehat{h}(\xi) \Big) e^{- |\xi|^{2} } d\xi \notag
   \\
    &=&0, \end{eqnarray}
and    \begin{equation}\label{proposition-integral-proof}
\int_{\mathbb{R}^{q} } \frac{e^{-|\underline{y}|^{2}}}{|\underline{y}|^{q-1}}\cosh(2|\underline{y}| |\xi|)d\underline{y}
   =|\mathbb{S}| \int_{0}^{+\infty}  e^{-r^{2}}  \cosh(2r |\xi|)dr=\frac{\sqrt{\pi}}{2}|\mathbb{S}|e^{|\xi|^{2}}. \end{equation}
Hence, we have
\begin{eqnarray*}
 & &\langle U_{p,q} [f], U_{p,q} [h]\rangle_{L^{2}(\mathbb{R}^{p+q+1}, d\mu)\otimes \mathbb{C}_{p+q}}
 \\
&=& \frac{2}{\sqrt{\pi}}\frac{1}{ |\mathbb{S}|} \int_{\mathbb{R}^{p+q+1} }   \big(U_{p,q} [f](x+\underline{y}), U_{p,q} [h](x+\underline{y})\big)  \frac{e^{-|\underline{y}|^{2}}}{|\underline{y}|^{q-1}} dxd\underline{y}
\\
 &\xlongequal{\text{Lemma \ref{remark-Segal-Bargmann-transform}}}&\frac{2}{\sqrt{\pi}}\frac{1}{ |\mathbb{S}|}\frac{ 1}{ (2\pi)^{ p+1 } } \int_{\mathbb{R}^{p+q+1} } \frac{e^{-|\underline{y}|^{2}}}{|\underline{y}|^{q-1}} dxd\underline{y} \int_{\mathbb{R}^{2(p+1)} } \big( e(x+\underline{y}, \xi) e^{-\frac{|\xi|^{2}}{2}}   \widehat{f}(\xi),
 \\
 && e(x+\underline{y},  \zeta) e^{-\frac{|\zeta|^{2}}{2}}   \widehat{h}(\zeta) \big)d\xi d\zeta
 \\
  &=&\frac{2}{\sqrt{\pi}}\frac{1}{ |\mathbb{S}|}\frac{1}{ (2\pi)^{ p+1 } } \int_{\mathbb{R}^{p+q+1} } \frac{e^{-|\underline{y}|^{2}}}{|\underline{y}|^{q-1}} dxd\underline{y} \int_{\mathbb{R}^{2(p+1)} } \big(   \widehat{f}(\xi),
 \\
 && e(x+\underline{y}, \xi)^{\dag} e(x+\underline{y},  \zeta)  \widehat{h}(\zeta) \big) e^{-\frac{|\xi|^{2}+|\zeta|^{2}}{2}} d\xi d\zeta
\\
  &\xlongequal{\text{ (\ref{proposition-e-proof})}} &\frac{2}{\sqrt{\pi}}\frac{1}{ |\mathbb{S}|}\frac{ 1}{ (2\pi)^{ p+1 } } \int_{\mathbb{R}^{p+q+1} } \frac{e^{-|\underline{y}|^{2}}}{|\underline{y}|^{q-1}} dxd\underline{y} \int_{\mathbb{R}^{2(p+1)} } \big(   \widehat{f}(\xi),
 \\
 && e(x, \zeta-\xi) e(\underline{y}, \xi)e(\underline{y}, \zeta)   \widehat{h}(\zeta) \big) e^{-\frac{|\xi|^{2}+|\zeta|^{2}}{2}} d\xi d\zeta
 \\
   &=&\frac{2}{\sqrt{\pi}}\frac{1}{ |\mathbb{S}|}\frac{ 1}{ (2\pi)^{ p+1 } } \int_{\mathbb{R}^{p+q+1} \times \mathbb{R}^{2(p+1)}} \frac{e^{-|\underline{y}|^{2}}}{|\underline{y}|^{q-1}} \big(   \widehat{f}(\xi),
   \\
   &&e(x, \zeta-\xi) e(\underline{y}, \xi)e(\underline{y}, \zeta)   \widehat{h}(\zeta) \big) e^{-\frac{|\xi|^{2}+|\zeta|^{2}}{2}} d\xi d\zeta dxd\underline{y}
   \\
   &=& \frac{2}{\sqrt{\pi}}\frac{1}{ |\mathbb{S}|} \int_{\mathbb{R}^{p+q+1} } \frac{e^{-|\underline{y}|^{2}}}{|\underline{y}|^{q-1}} \big(   \widehat{f}(\xi), e(\underline{y}, \xi)^{2}   \widehat{h}(\xi) \big) e^{- |\xi|^{2} } d\xi   d\underline{y}
   \\
   &\xlongequal{\text{(\ref{proposition-e-4})}} & \frac{2}{\sqrt{\pi}}\frac{1}{ |\mathbb{S}|} \int_{\mathbb{R}^{p+q+1} } \frac{e^{-|\underline{y}|^{2}}}{|\underline{y}|^{q-1}} \big(   \widehat{f}(\xi),
  e(2\underline{y}, \xi)    \widehat{h}(\xi) \big) e^{- |\xi|^{2} } d\xi   d\underline{y}
  \\
   &\xlongequal{\text{Proposition \ref{example-e}}}& \frac{2}{\sqrt{\pi}}\frac{1}{ |\mathbb{S}|} \int_{\mathbb{R}^{p+q+1} } \frac{e^{-|\underline{y}|^{2}}}{|\underline{y}|^{q-1}} \big(   \widehat{f}(\xi),
   \cosh(2|\underline{y}| |\xi|)     \widehat{h}(\xi) \big) e^{- |\xi|^{2} } d\xi   d\underline{y}
   \\
   & & +\frac{2}{\sqrt{\pi}}\frac{1}{ |\mathbb{S}|} \int_{\mathbb{R}^{p+q+1} } \frac{e^{-|\underline{y}|^{2}}}{|\underline{y}|^{q-1}} \Big(   \widehat{f}(\xi),
    i \frac{\underline{y}}{|\underline{y}|}\frac{\xi}{|\xi|} \sinh(2|\underline{y}| |\xi|)    \widehat{h}(\xi) \Big) e^{- |\xi|^{2} } d\xi   d\underline{y}
    \\
   &\xlongequal{\text{(\ref{proposition-integral-0-proof})}}& \frac{2}{\sqrt{\pi}}\frac{1}{ |\mathbb{S}|} \int_{\mathbb{R}^{p+1} } \big(   \widehat{f}(\xi),
   \widehat{h}(\xi) \big) e^{- |\xi|^{2} } d\xi   \int_{\mathbb{R}^{q} } \frac{e^{-|\underline{y}|^{2}}}{|\underline{y}|^{q-1}}\cosh(2|\underline{y}| |\xi|)d\underline{y}
   \\
   &\xlongequal{\text{(\ref{proposition-integral-proof})}}& \int_{\mathbb{R}^{p+1} } \big(   \widehat{f}(\xi),
   \widehat{h}(\xi) \big) d\xi
   \\
   &=& \langle \widehat{f}, \widehat{h} \rangle_{L^{2}(\mathbb{R}^{p+1}, d\xi)\otimes \mathbb{C}_{p+q}}
   \\
   & \xlongequal{\text{Theorem \ref{Plancherel-theorem}}}& \langle  f, h \rangle_{L^{2}(\mathbb{R}^{p+1}, dx)\otimes \mathbb{C}_{p+q}},
    \end{eqnarray*}
  as desired.
 \end{proof}

In the one variable case, for  $k\in \mathbb{N}_{0}$, $k$-th   Hermite polynomials $H_k(x)$ are defined by
$$H_k(x)= (-1)^{k}e ^{x^{2}}  \frac{d}{dx^{k}}e ^{-x^{2}}.$$
In alternative, they are also given by the formula
$$e^{2xy-y^{2}}=\sum_{k=0}^{\infty} H_k(x) \frac{y^{k}}{k!}.$$
It is well-known that
$$\int_{\mathbb{R}} H_k(x)H_l(x) e^{-x^{2}}dx = \sqrt{\pi}2^{k}k! \delta_{k,l},$$
where $\delta_{k,l}$ is the Kronecker delta function.

In the several  variables case, let now $\{H_{\mathrm{k}}, \mathrm{k} \in \mathbb{N}_{0}^{p+1}\}$ denote the orthogonal basis of $L^2(\mathbb{R}^{p+1},e^{-|x|^{2}}dx)$ consisting of the Hermite polynomials on $\mathbb{R}^{p+1}$ defined as
$$H_{\mathrm{k}}(x)=H_{k_{0}}(x_{0})H_{k_{1}}(x_{1})\cdots H_{k_{p}}(x_{p}),$$
where $\mathrm{k}=(k_0,k_1,\ldots,k_p).$

\begin{lemma}\label{factor-onto-orth}
Define, for $\mathrm{k} \in \mathbb{N}_{0}^{p+1}$,
$$\varphi_{\mathrm{k}}(x)=H_{\mathrm{k}}(x)e^{-\frac{|x|^{2}}{2}}, \quad x\in \mathbb{R}^{p+1}.$$
Then $\{\varphi_{\mathrm{k}}, \mathrm{k} \in \mathbb{N}_{0}^{p+1}\} \subset \mathcal{S}(\mathbb{R}^{p+1})$  is an orthogonal basis for $L^2(\mathbb{R}^{p+1},dx)$.
\end{lemma}
\begin{proof}
By their definition and by the orthogonality of Hermite polynomials in one variable,
  we deduce that, for $\mathrm{k}=(k_0,k_1,\ldots,k_p)$ and $ \mathrm{l}=(l_0,l_1,\ldots,l_p)$,
$$\int_{\mathbb{R}^{p+1}} H_\mathrm{k}(x)H_\mathrm{l}(x) e^{-|x|^{2}}dx = \pi^{(p+1)/2}2^{|\mathrm{k}|} \mathrm{k}! \delta_{\mathrm{k},\mathrm{l}},$$
where $|\mathrm{k}|=k_0+k_1+\ldots+k_p, \mathrm{k}!=k_0!k_1!\ldots k_p!,$ and $\delta_{\mathrm{k},\mathrm{l}}=\delta_{k_0,l_0}\delta_{k_1,l_1}\ldots\delta_{k_p,l_p}$.

\end{proof}

\begin{lemma}\label{factor-onto}
For $\mathrm{k} \in \mathbb{N}_{0}^{p+1}$, we have
$$ \exp(\frac{\Delta_{x}}{2})\varphi_{\mathrm{k}}(x)=2 ^{-\frac{p+1}{2}} x^{\mathrm{k}} e^{-\frac{|x|^{2}}{4}}, \quad x\in \mathbb{R}^{p+1}.$$
\end{lemma}
 \begin{proof}
 By definition, we have
 \begin{eqnarray*}
 & & e^{\frac{|x|^{2}}{4}} \exp\frac{\Delta_{x}}{2}\varphi_{\mathrm{k}}(x)  \\
&=& \frac{1}{(2\pi )^{(p+1)/2}}  e^{\frac{|x|^{2}}{4}} \int_{\mathbb{R}^{p+1} } e^{-\frac{|x-\xi|^{2}}{2}} H_{\mathrm{k}}(\xi)e^{-\frac{|\xi |^{2}}{2}}d\xi
\\
&=& \frac{1}{(2\pi )^{(p+1)/2}} \int_{\mathbb{R}^{p+1} } e^{ (x,\xi)-\frac{|x|^{2}}{4}} H_{\mathrm{k}}(\xi)e^{-|\xi |^2}d\xi
\\
&=& \frac{1}{(2\pi )^{(p+1)/2}} \int_{\mathbb{R}^{p+1} }  \sum_{\mathrm{l}\in \mathbb{N}_{0}^{p+1}}  H_\mathrm{l}( \xi ) \frac{(x/2)^{\mathrm{k}}}{\mathrm{k}!} H_{\mathrm{k}}(\xi)e^{-|\xi |^2}d\xi
\\
&=& 2^{-\frac{ p+1} {2}} x^{\mathrm{k}},
\end{eqnarray*}
as desired.
 \end{proof}

 \begin{lemma} For $\mathrm{k} \in \mathbb{N}_{0}^{p+1}$, define
$$\psi_{\mathrm{k}}(\bx):=2^{(p+1)/2} U_{p,q}[\varphi_{\mathrm{k}}](\bx)\equiv CK[x^{\mathrm{k}} e^{-\frac{|x|^{2}}{4}}](\bx), \quad \bx\in \mathbb{R}^{p+q+1}.$$
Then, the set
$$\big\{ \psi_{\mathrm{k}} e_{A}: \mathrm{k} \in \mathbb{N}_{0}^{p+1}, A\subseteq \{1,2,\ldots, p+q\} \big\}$$
is an orthogonal basis for the Hilbert space $\mathcal{GSM}L^{2}(\mathbb{R}^{p+q+1}, d\mu)$.
\end{lemma}

 \begin{proof}
 Since $x^{\mathrm{k}} e^{-\frac{|x|^{2}}{4}}\in \mathcal{S}(\mathbb{R}^{p+1})$ for all $\mathrm{k} \in \mathbb{N}_{0}^{p+1}$, $\psi_{\mathrm{k}}(\bx)$ are well-defined on $\mathbb{R}^{p+q+1}$  from Proposition \ref{Ck-entire-condition-Schwarz} and  Lemma \ref{factor-onto}.

 The   conclusion that $\big\{ \psi_{\mathrm{k}} e_{A} \big\}$ is an orthogonal basis for   $\mathcal{GSM}L^{2}(\mathbb{R}^{p+q+1}, d\mu)$ follows directly from Lemmas \ref{isometry-Schwarz} and \ref{factor-onto-orth} and this completes the proof.

 \end{proof}

 Now we are in a position to prove Theorem \ref{main-Coherent-state-transforms}.
  \begin{proof}[Proof of Theorem \ref{main-Coherent-state-transforms}] By the density of  $ \mathcal{S}(\mathbb{R}^{p+1})$ in $L^{2}(\mathbb{R}^{p+1}, dx)$ we deduce  by
Lemma  \ref{isometry-Schwarz}  that $U_{p,q}$  is an isometry onto its image which is, therefore, closed in $L^{2}(\mathbb{R}^{p+q+1}, d\mu)\otimes \mathbb{C}_{p+q}$.  In view of  Lemma  \ref{partial-slice-monogenic}, we get $ U_{p,q}(L^{2}(\mathbb{R}^{p+1}, dx)\otimes \mathbb{C}_{p+q})\subseteq  \mathcal{GSM} L^{2}(\mathbb{R}^{p+q+1})$. Hence, $ U_{p,q}(L^{2}(\mathbb{R}^{p+1}, dx)\otimes \mathbb{C}_{p+q})$ is a Hilbert space of  generalized partial-slice
monogenic functions.

 Now let us  prove that the map  $U_{p,q}$  is onto. Note that, by Theorem \ref{Identity-theorem}, the restriction of $f\in \mathcal{GSM}(\mathbb{R}^{p+q+1})$  to the hyperplane $\underline{y}=0$, $f_{0}(x) (x\in\mathbb{R}^{p+1})$ determines $f(\bx) (\bx\in\mathbb{R}^{p+q+1})$ uniquely. Since  both generalized  partial-slice monogenic functions (see Theorem  \ref{Taylor-theorem}) and the factor $e^{\frac{|x|^{2}}{4}}$  have a Taylor series at origin with infinite radius of convergence, it follows that $f_0(x)$ can be expressed uniquely in the form
 $$f_0(x)e^{\frac{|x|^{2}}{4}} =\sum_{A\subseteq \{1,2,\ldots, p+q\}}\sum_{\mathrm{k} \in \mathbb{N}_{0}^{p+1}} x^{\mathrm{k}} a_{\mathrm{k},A}e_{A}, \quad x\in\mathbb{ R}^{p+1},a_{\mathrm{k},A}\in \mathbb{C},$$
 or
$$f_0(x)=\sum_{\mathrm{k}\in \mathbb{N}_{0}^{p+1}} x^{\mathrm{k}} e^{-\frac{|x|^{2}}{4}} a_{\mathrm{k}}, \quad x\in\mathbb{ R}^{p+1},$$
where $$ a_{\mathrm{k}}=\sum_{A\subseteq \{1,2,\ldots, p+q\}}  a_{\mathrm{k},A}e_{A}\in \mathbb{C}_{p+q}.$$
Hence, by Corollary \ref{slice-Cauchy-Kovalevsky-extension-entire-Identity}, we have
     $$f(\bx)=CK[f_{0}](\bx)=\sum_{m=0}^{+\infty}\sum_{  \mathrm{k}\in \mathbb{N}_{0}^{p+1}} \frac{1}{m!}( \underline{y}D_{x})^{m} (x^{ \mathrm{k}} e^{-\frac{|x|^{2}}{4}}) a_{ \mathrm{k}},$$
where convergent power series can be differentiated term by term. This series converges absolutely in $\mathbb{ R}^{p+q+1}$   so that
the two summations can be interchanged giving
   $$f(\bx)=\sum_{ \mathrm{k}\in \mathbb{N}_{0}^{p+1}}\sum_{m=0}^{+\infty} \frac{1}{m!}( \underline{y}D_{x})^{m} (x^{ \mathrm{k}} e^{-\frac{|x|^{2}}{4}}) a_{k}=\sum_{ \mathrm{k}\in \mathbb{N}_{0}^{p+1}} CK[x^{ \mathrm{k}} e^{-\frac{|x|^{2}}{4}}](\bx) a_{ \mathrm{k}}.$$
 By Lemma \ref{factor-onto},   all finite  sums
$$\sum_{ \mathrm{k} \in \mathbb{N}_{0}^{p+1},  \mathrm{k} \ \mathrm{  finite}} CK[x^{ \mathrm{k}} e^{-\frac{|x|^{2}}{4}}] a_{ \mathrm{k}}
=\sum_{ \mathrm{k}\in \mathbb{N}_{0}^{p+1},  \mathrm{k} \ \mathrm{  finite}}\psi_{ \mathrm{k}}(\bx)a_{ \mathrm{k}} $$
 are in the image   $U_{p,q}(L^{2}(\mathbb{R}^{p+1}, dx)\otimes \mathbb{C}_{p+q})$ which is closed. Moreover,  from Lemmas \ref{isometry-Schwarz} and \ref{factor-onto-orth}, we obtain that the condition $f\in L^{2}(\mathbb{R}^{p+q+1}, d\mu)\otimes \mathbb{C}_{p+q}$ is equivalent
to
$$\sum_{ \mathrm{k}\in \mathbb{N}_{0}^{p+1}}\varphi_{ \mathrm{k}}(\bx)a_{ \mathrm{k}}\in L^{2}(\mathbb{R}^{p+1}, dx)\otimes \mathbb{C}_{p+q},$$
which completes the proof.
\end{proof}

\section{A possible quantum mechanical interpretations}

 Define the position and momentum operators  in $L^2(\mathbb{R}^{p+1}, dx)$  by
$$X_k[f](x)= x_{k}f(x), \  P_k[f](x)=-i \frac{\partial}{\partial x_{k}} f(x), \quad k=0,1,\ldots,p,$$
  respectively.
Note that $X_k$  and $P_k$ are densely defined and (unbounded) self-adjoint operators.

 With the Segal-Bargmann  transform defined in  (\ref{Segal-Bargmann-transform-def-holo}),  the  Schr\"{o}dinger  representation in quantum mechanics can be  written as (see \cite[Theorem 6.4]{Hall})
\begin{equation}\label{representation-Hall}
U\circ (X_k-i P_k)\circ U^{-1} [f](z)=z_{k} f(z), \quad k=0,1,\ldots,p.\end{equation}

For  $f\in L^{2}(\mathbb{R}^{p+1}, dx)\otimes \mathbb{C}_{p+q}$, we define
$$X[f](x)=\sum_{i=0}^{p}e_k X_k[f](x)= xf(x), $$
and
$$ P[f](x)=\sum_{i=0}^{p}e_k P_k[f](x)=-i D_{x} f(x),$$
where  $xf(x)$ is the left multiplication by the paravector variable $x\in\mathbb R^{p+1}$ and $D_{x}$ is given by  (\ref{Dxx}).

Using the Segal-Bargmann transform given by Definition  \ref{definition-Segal-Bargmann-transform}, we can now establish a generalization of (\ref{representation-Hall}), which generalizes \cite[Proposition 5.1]{Qian-16} for slice monogenic functions and \cite[Theorem 5.1]{Qian-17}  for    monogenic functions.
 \begin{theorem}\label{representation}
The unitary  map $U_{p,q}$ induces a representation   on the
Hilbert space of  generalized partial-slice monogenic   functions $\mathcal{GSM} L^{2}(\mathbb{R}^{p+q+1}, d\mu)$
 $$U_{p,q}\circ (X-i P)\circ U_{p,q}^{-1} [f](\bx)=CK[xf(x)](\bx).$$
 \end{theorem}
\begin{proof}
By the density of $ \mathcal{S}(\mathbb{R}^{p+1})\otimes \mathbb{C}_{p+q}$ in $L^{2}(\mathbb{R}^{p+1}, dx)\otimes \mathbb{C}_{p+q}$, we consider only $h\in \mathcal{S}(\mathbb{R}^{p+1})\otimes \mathbb{C}_{p+q}$ and let $f=U_{p,q}[h]\in \mathcal{GSM} L^{2}(\mathbb{R}^{p+q+1}, d\mu)$. Then we consider
$$(X-i P)[h](x)=\frac{1}{ (2\pi)^{(p+1)/2} }\int_{\mathbb{R}^{p+1} }   (x-i\xi) e^{i\langle x, \xi \rangle}    \widehat{h}(\xi)d\xi.$$
From (\ref{representation-Hall}), it follows that
$$\exp(\frac{\Delta_{x}}{2}) [(x-i\xi) e^{i\langle x, \xi \rangle}]=x \exp (\frac{\Delta_{x}}{2})  [e^{i\langle x, \xi \rangle}].$$
Hence, we have
\begin{eqnarray*}
  U_{p,q}\circ(X-i P)\circ U_{p,q}^{-1} [f]
&=&CK\circ \exp\frac{\Delta_{x}}{2} \circ (X-i P)[h(x)]
\\
&=&CK \Big[\frac{1}{ (2\pi)^{(p+1)/2} }\int_{\mathbb{R}^{p+1} }   x \exp\frac{\Delta_{x}}{2} [ e^{i\langle x, \xi \rangle}  ] \widehat{h}(\xi)d\xi\Big]
\\
 &=& CK  [   x \exp\frac{\Delta_{x}}{2}  h(x)  ].
 \end{eqnarray*}
 Recalling that $f(x)=\exp\frac{\Delta_{x}}{2}  [h](x)$ for all $x\in \mathbb{R}^{p+1}$, we finally get
$$ U_{p,q}\circ(X-i P)\circ U_{p,q}^{-1} [f](\bx)= CK  [   x f(x)  ] (\bx),$$
which concludes the proof.
\end{proof}





\vskip 10mm
\end{document}